\documentclass[11pt]{article}
\usepackage{amsmath, afterpage, pdflscape, changepage, multirow, multicol, cleveref, hhline, array, xcolor, verbatim, graphicx, url, fancyhdr, etoolbox, bigstrut, cite, setspace, float, xr, amssymb, mathtools, capt-of, amsthm, bm, hhline, blkarray, pbox, pdfpages, bbm, extarrows, mathrsfs, mathabx}
\usepackage[font=small,labelfont=bf]{caption}
\usepackage[hang]{footmisc}
\usepackage[USenglish]{babel}
\usepackage[USenglish]{isodate}
\usepackage[hmargin=1.5cm,top=1.5cm,bottom=1.5cm]{geometry}
\usepackage[inline]{enumitem}

\let\svthefootnote\thefootnote

\newcommand\blankfootnote[1]{%
  \let\thefootnote\relax\footnotetext{#1}%
  \let\thefootnote\svthefootnote%
}

\makeatletter
\newcommand\blfootnote[1]{%
  \begingroup
  \renewcommand{\@makefntext}[1]{\noindent\makebox[0.5em][r]#1}
  \renewcommand\thefootnote{}\footnote{#1}%
  \addtocounter{footnote}{-1}%
  \endgroup
}
\makeatother

\newtheorem{theorem}{Theorem}[section]
\newtheorem{lemma}[theorem]{Lemma}
\newtheorem{proposition}[theorem]{Proposition}
\newtheorem{corollary}[theorem]{Corollary}

\theoremstyle{remark}
\newtheorem{definition}[theorem]{Definition}

\newtheorem{remark}[theorem]{Remark}
\newtheorem{remarks}[theorem]{Remarks}
\newtheorem{notation}[theorem]{Notation}

\def\NN{{\mathbb N}}

\def\ZZ{{\mathbb Z}}

\def\C{{\mathcal C}}
\def\D{{\mathcal D}}
\def\E{{\mathcal E}}
\def\F{{\mathcal F}}

\def\N{{\mathcal N}}

\def\SC{{\mathcal S}}
\def\T{{\mathcal T}}

\def\X{{\mathcal X}}
\def\Z{{\mathcal Z}}

\def\sX{{\mathscr X}}

\newcommand{\Coin}{\operatorname{Coin}}

\newcommand{\Dom}{\operatorname{Dom}}

\newcommand{\Image}{\operatorname{Im}}

\newcommand{\Supp}{\operatorname{Supp}}

\newcommand{\powerset}{\raisebox{.15\baselineskip}{\Large\ensuremath{\wp}}}

\newcommand{\defhypergraph}{X := \big(V(X),E(X)\big)}

\makeatletter
\def\namedlabel#1#2{\begingroup
   \def\@currentlabel{#2}%
   \label{#1}\endgroup
}
\makeatother

\title{\LARGE Transformations on hypergraph families}

\author{\normalsize Sean T. Vittadello}

\fancypagestyle{specialfooter}{%
\fancyhf{}

\fancyfoot[R]{\footnotesize \emph{\today}}
}

\begin{document}
\makeatletter
\begin{titlepage}
\thispagestyle{specialfooter}
\blankfootnote{\textbf{Address}: School of Mathematics and Statistics \& School of BioSciences, The University of Melbourne, Parkville, Victoria 3010, Australia}
\blankfootnote{\textbf{E-mail address}: sean.vittadello@unimelb.edu.au}
\blankfootnote{\textbf{Key words and phrases}: Hypergraph transformation, hypergraph family, function-based transformation, quotient transformation}
\centering
\@title\\
\vspace{1.5cm}
\@author\\
\vspace{1.5cm}
\cleanlookdateon
\vspace{0mm}
\begin{abstract}
\begin{normalsize}\noindent
We present a new general theory of function-based hypergraph transformations on finite families of finite hypergraphs. A function-based hypergraph transformation formalises the action of structurally modifying hypergraphs from a family in a consistent manner. The mathematical form of the transformations facilitates their analysis and incorporation into larger mathematical structures, and concurs with the function-based nature of modelling in the physical world. Since quotients of hypergraphs afford their simplification and comparison, we also discuss the notion of a quotient hypergraph transformation induced by an equivalence relation on the vertex set of a hypergraph family. Finally, we demonstrate function-based hypergraph transformations with two fundamental classes of examples involving the addition or deletion of hyperedges or hypergraphs.
\end{normalsize}
\end{abstract}
\end{titlepage}
\makeatother

\section{Introduction}
A hypergraph transformation acts on a given hypergraph to yield a new hypergraph with specific modifications, such as the addition or deletion of a hyperedge or the replacement of a subhypergraph by another subhypergraph. Hypergraph transformations therefore provide a formal description of the structural relationship between two hypergraphs, more general than the transformations induced by hypergraph homomorphisms, and find broad application in both pure and applied mathematics including spectral graph theory \cite{Merris1998}, graph similarity \cite{Pfaltz2018}, the maximum stable set problem \cite{Lozin2011}, network analysis \cite{Blinov2006,Bunimovich2012,Andersen2013}, engineering design \cite{Voss2023}, and theoretical computer science \cite{Rozenberg1997,Ehrig2006,Heckel2006}. The form in which a hypergraph transformation is expressed depends on the application, and may be descriptive, rule based, or function based. Rule-based hypergraph transformations, which modify hypergraphs algorithmically, have received the most attention because of their utility in theoretical computer science \cite{Rozenberg1997}.


Hypergraphs are increasingly finding application in many areas of natural science as models of complex systems with higher-order interactions \cite{Klamt2009,Feng2021}, where hypergraph transformations can represent dynamic system behaviour \cite{Rossello2004,Yadav2004,Rossello2005}. Function-based representations of physical phenomena are fundamental in natural science, so expressing hypergraph transformations as functions aligns with the general mathematical formalism employed in mathematical modelling. Further, function-based hypergraph transformations can be readily incorporated into larger mathematical objects with additional structure, allowing for the development of more detailed mathematical models and the application of new techniques for mathematical analysis: for example, a partially ordered set of hypergraph transformations may represent a hierarchical system of dynamic processes. While some instances of function-based hypergraph transformations exist \cite{Pfaltz2018}, no general theory for function-based hypergraph transformations that is suitable within the context of natural science has been described in the literature.

In this article we introduce and develop a new approach to function-based hypergraph transformations, where each transformation is defined on a finite family of finite hypergraphs and acts consistently on all hypergraphs in its domain. Given the importance of the concept of a quotient hypergraph, which can assist with the simplification and comparison of hypergraphs, we also consider the notion of a quotient hypergraph transformation. We illustrate the general theory with two fundamental classes of examples involving the addition or deletion of hyperedges and the addition or deletion of hypergraphs.


\section{Preliminaries}\label{sec:Preliminaries}
In this section we discuss preliminary notation, definitions, and results. We begin with a review of hypergraphs, their substructures, and connectivity.

\begin{notation}[\textbf{Sets}]
Denote the set of positive integers by $\NN$, the set of nonnegative integers by $\NN_0$, and $[n] := \{\, m \le n \mid m \in \NN \,\}$ for $n \in \NN_0$. Given a set $S$ we denote by $\powerset (S)$ the power set of $S$.
\end{notation}

\begin{definition}[\textbf{Hypergraphs}]
A \emph{hypergraph} is a 2-tuple $\defhypergraph$ of finite sets where $V(X)$ is the \emph{vertex set} and $E(X) \subseteq \powerset \big(V(X)\big)$ is the \emph{hyperedge set}. We exclude the empty hyperedge, hence $\lvert e \rvert \ge 1$, and allow multiple hyperedges. A \emph{loop} is a hyperedge that is a singleton set, so \emph{allowing loops} (resp. \emph{disallowing loops}) assumes $\lvert e \rvert \ge 1$ (resp. $\lvert e \rvert \ge 2$) for all $e \in E(X)$. If $V(X) = \emptyset$, so that $E(X) = \emptyset$, then $X$ is called the \emph{null hypergraph} and is denoted by $\mathcal{N}$.

Given a hypergraph $X$, a \emph{vertex labelling} (resp. \emph{hyperedge labelling}) of $X$ is an injective function $\ell_V \colon V(X) \to L_V$ (resp. $\ell_E \colon E(X) \to L_E$), where $L_V$ and $L_E$ are nonempty disjoint sets of labels. A vertex-labelled hypergraph therefore has all vertices uniquely labelled, and a hyperedge-labelled hypergraph has all hyperedges, including multiple hyperedges, uniquely labelled.

We henceforth fix nonempty disjoint sets of labels $L_V$ and $L_E$, and denote by $\sX$ the universe of all vertex- and hyperedge-labelled hypergraphs over $L_V$ and $L_E$. We refer to $\sX$ as a \emph{hypergraph family}. Since we only consider vertex- and hyperedge-labelled hypergraphs in this work we refer to them simply as hypergraphs without mentioning the labelling functions explicitly. For any subset $\X \subseteq \sX$, denote $V(\X) := \{\, v \in V(X) \mid X \in \X \,\}$, $E(\X) := \{\, e \in E(X) \mid X \in \X \,\}$, and $\X^{\ast} := \X \setminus \{\N\}$. If $\X$ is a singleton set, say $\X = \{X\}$, then we may denote $\X$ simply by $X$.
\end{definition}

While we consider hypergraphs without directed hyperedges, they could readily be included. We now consider substructures and connectivity in hypergraphs \cite{Dewar2018}.

\begin{definition}[\textbf{Strong subhypergraphs}]
Let $X$, $Y \in \sX$. Then $X$ is a \emph{strong subhypergraph} of $Y$ if $V(X) \subseteq V(Y)$ and $E(X) \subseteq E(Y)$. In this case we say that $Y$ contains $X$, and if $X \ne Y$ then the containment is \emph{proper}. The strong subhypergraph $X$ of $Y$ is \emph{induced} by $V(X)$ if $E(X) = \big\{\, e \in E(Y) \mid e \subseteq V(X) \,\big\}$. Note that the null hypergraph $\N$ is a strong subhypergraph of every hypergraph. If $\F$ is a family of strong subhypergraphs of $Y$ then $X \in \F$ is \emph{maximal} in $\F$ if no hypergraph in $\F$ properly contains $X$.
\end{definition}

\begin{definition}[\textbf{Connectivity}]
Let $X \in \sX$. A \emph{walk} in $X$ is a nonempty alternating sequence $(v_0,e_1,v_1,\ldots,e_m,v_m)$ of vertices and hyperedges in $X$ such that:
\begin{enumerate*}[label=(\arabic*)]
\item $v_i \in V(X)$ for all $i \in [m] \cup \{0\}$;
\item $e_i \in E(X)$ for all $i \in [m]$;
\item $v_i$, $v_{i+1} \in e_{i+1}$ for all $0 \le i \le m-1$.
\end{enumerate*}
If the $m+1$ vertices are pairwise distinct and the $m$ hyperedges are pairwise distinct then the walk is a \emph{path}. Note that a trivial walk $(v_0)$ is a path.

Two vertices $v$, $w \in V(X)$ are \emph{connected} in $X$ if there exists a path in $X$ from $v$ to $w$. The hypergraph $X$ is \emph{connected} if it is nonnull and if every pair of vertices in $X$ is connected, otherwise $X$ is \emph{disconnected}. A \emph{connected component}, or simply \emph{component}, of $X$ is a connected strong subhypergraph of $X$ that is maximal in the family of connected strong subhypergraphs of $X$. Note that the null hypergraph is disconnected and has no components. We denote by $\C(X)$ the set of connected components of $X$. Further, for $\X \subseteq \sX$ we denote $\C(\X) := \bigcup_{X \in \X} \C(X)$.
\end{definition}

\begin{definition}[\textbf{Disjointness, union}]
Let $X$, $Y \in \sX$. Then $X$ and $Y$ are \emph{vertex disjoint}, or simply \emph{disjoint}, if $V(X) \cap V(Y) = \emptyset$ (which implies $E(X) \cap E(Y) = \emptyset$), and are \emph{component disjoint} if $\C(X) \cap \C(Y) = \emptyset$. Note that vertex disjointness implies component disjointness. Further, the \emph{hypergraph union} of $X$ and $Y$ is the hypergraph $X \cup Y$ with $V(X \cup Y) = V(X) \cup V(Y)$ and $E(X \cup Y) = E(X) \cup E(Y)$.
\end{definition}

We define a direct sum of hypergraphs, which corresponds to the similar notion for graphs \cite{Knuth1994}, and provides a convenient decomposition of hypergraphs.

\begin{definition}[\textbf{Direct sum of hypergraphs}]
Let $\{X_i\}_{i \in [m]} \subseteq \sX$ be a subset of pairwise disjoint hypergraphs with $m \in \NN_0$. The \emph{direct sum} of the hypergraphs $\{X_i\}_{i \in [m]}$, denoted by $X_1 \oplus \cdots \oplus X_m$, $\bigoplus_{i \in [m]} X_i$, or $\bigoplus \{X_i\}_{i \in [m]}$, is the hypergraph with $V(X_1 \oplus \cdots \oplus X_m) := \bigcup_{i \in [m]} V(X_i)$ and $E(X_1 \oplus \cdots \oplus X_m) := \bigcup_{i \in [m]} E(X_i)$.
\end{definition}

Note the following: the direct sum of hypergraphs is independent of the order of the hypergraph summands; $X \in \sX$ has the direct sum decomposition $X = \bigoplus_{C \in \C(X)} C = \bigoplus \C(X)$; if $m = 0$ then the set of pairwise disjoint hypergraphs is empty, so the direct sum has no summands, hence the direct sum is the null hypergraph $\N$.

We also define a notion of the \emph{direct difference} of hypergraphs.

\begin{definition}[\textbf{Direct difference of hypergraphs}]
Let $X$, $Y$, $Z \in \sX$ where $Y$ and $Z$ are disjoint and $X = Y \oplus Z$. The \emph{direct difference} of the hypergraphs $X$ and $Z$, denoted by $X \ominus Z$, is $Y = X \ominus Z$. Since a direct sum of hypergraphs is independent of the order of the hypergraph summands we also have $Z = X \ominus Y$.
\end{definition}

\begin{proposition}\label{prop:DirectDiff}
Let $X$, $Y$, $Z \in \sX$.
\begin{enumerate}[label=(\arabic*)]
\item If $Y$ and $Z$ are disjoint and $X = Y \oplus Z$ then $X = (X \ominus Z) \oplus Z$ and $Y = (Y \oplus Z) \ominus Z$. \label{prop:DirectDiff1}
\item If $Y$ and $Z$ are disjoint, $\C(Y) \subseteq \C(X)$, and $\C(Z) \subseteq \C(X)$ then $(X \ominus Y) \ominus Z = (X \ominus Z) \ominus Y$. \label{prop:DirectDiff2}
\item If $Y$ is disjoint with both $X$ and $Z$, and $\C(Z) \subseteq \C(X)$ then $(X \oplus Y) \ominus Z = (X \ominus Z) \oplus Y$. \label{prop:DirectDiff3}
\end{enumerate}
\end{proposition}

\begin{proof}
\ref{prop:DirectDiff1} Since $Y = X \ominus Z$ it follows that $X = Y \oplus Z = (X \ominus Z) \oplus Z$ and $Y = X \ominus Z = (Y \oplus Z) \ominus Z$.

\ref{prop:DirectDiff2} Let $W := \bigcup \C(X) \setminus \C(Y \oplus Z)$ so that $X = W \oplus (Y \oplus Z) = W \oplus (Z \oplus Y)$, then $(X \ominus Z) \ominus Y = W = (X \ominus Y) \ominus Z$.

\ref{prop:DirectDiff3} Let $W := \bigcup \C(X \oplus Y) \setminus \C(Z)$, so that $X \oplus Y = W \oplus Z$, and then $W = (X \oplus Y) \ominus Z$. Now, $X = (W \oplus Z) \ominus Y$, hence $X \ominus Z = \big((W \oplus Z) \ominus Y\big) \ominus Z = \big((W \oplus Z) \ominus Z\big) \ominus Y = W \ominus Y$, where the second equality follows from Part~\ref{prop:DirectDiff2} of this proposition and the third equality follows from Part~\ref{prop:DirectDiff1} of this proposition, therefore $W = (X \ominus Z) \oplus Y$. We conclude that $(X \oplus Y) \ominus Z = (X \ominus Z) \oplus Y$.
\end{proof}

Our definition of a quotient hypergraph is standard.

\begin{definition}[\textbf{Quotient hypergraph}]\label{def:QuotHyp}
Suppose $X \in \sX$ and $R$ is an equivalence relation on $V(X)$. The \emph{quotient hypergraph} of $X$ under $R$, denoted $X/R$, is the hypergraph where:
\begin{enumerate}[label={(\arabic*)}]
\item $V(X/R) = \big\{\, [v]_R \mid v \in V(X)\,\big\}$, where $[\cdot]_R$ denotes the equivalence class under $R$.
\item $E(X/R) \subseteq \powerset \big(V(X/R)\big)$, where $\big\{[v_i]_R\big\}_{i=1}^n \in E(X/R)$ for $n \in \NN$ (resp. $n\ge2$ if we disallow loops) if and only if there exists $e \in E(X)$ such that
	\begin{enumerate*}[label={(\arabic*)}]
	\item $e \cap [v_i]_R \ne \emptyset$ for all $1 \le i \le n$, and
	\item $e \subseteq \bigcup_{i=1}^n [v_i]_R$.
	\end{enumerate*}
\end{enumerate}
The map $\theta \colon V(X) \to V(X/R)$ such that $\theta(v) = [v]_R$ for $v \in V(X)$ is the \emph{projection}.
\end{definition}

We will have occasion to construct hypergraphs that are related to quotient hypergraphs, however with vertices consisting of equivalence classes in more general underlying vertex sets. For this we introduce the notions of a \emph{vertex-augmented hypergraph} and a \emph{vertex-augmented quotient hypergraph}.

\begin{definition}[\textbf{Vertex-augmented hypergraph and quotient hypergraph}]\label{def:VaugH}
Suppose $F \subseteq V(\sX)$ and $X \in \sX$ with $V(X) \subseteq F$. The hypergraph $Z := \big(F,E(X)\big)$ is the \emph{vertex-augmented hypergraph} of $X$ with respect to $F$. If $R_F$ is an equivalence relation on $F$, with corresponding equivalence classes denoted $[\cdot]_{R_F}$, then the \emph{vertex-augmented quotient hypergraph} of $X$ under $R_F$, denoted $X//R_F$, is the hypergraph where $V(X//R_F) := \big\{\, [v]_{R_F} \mid v \in V(X) \,\big\}$ and $E(X//R_F) := E(Z/R_F)$. The map $\theta \colon V(X) \to V(X//R_F)$ such that $\theta(v) = [v]_{R_F}$ for $v \in V(X)$ is the \emph{projection}.
\end{definition}

The following proposition establishes a canonical isomorphism between a given vertex-augmented quotient hypergraph and the corresponding quotient hypergraph under the restricted equivalence relation. We formalise this result to show the explicit relationship between the two forms of quotient hypergraphs.

\begin{proposition}\label{prop:IsoQuot}
Suppose $F \subseteq V(\sX)$, $X \in \sX$ with $V(X) \subseteq F$, $R_F$ is an equivalence relation on $F$, and $R$ is the restriction of $R_F$ to $V(X)$. Then the vertex map $\phi \colon V(X/R) \to V(X//R_F)$ given by $\phi \big([v]_R\big) = [v]_{R_F}$ for $[v]_R \in V(X/R)$ is an isomorphism from the quotient hypergraph $X/R$ onto the vertex-augmented quotient hypergraph $X//R_F$, where the inverse map $\phi^{-1}$ satisfies $\phi^{-1} \big([w]_{R_F}\big) = [v]_R$ for $[w]_{R_F} \in V(X//R_F)$ and any representative $v \in [w]_{R_F} \cap V(X)$.
\end{proposition}

\begin{proof}
$\phi$ is well defined and injective since $[v]_R = [w]_R$ if and only if $[v]_{R_F} = [w]_{R_F}$, for $v$, $w \in V(X)$, and it follows from the definition of $V(X//R_F)$ that $\phi$ is surjective. If $f \subseteq V(X/R)$ then $f \in E\big(X/R\big)$ if and only if $\phi(f) \in E(X//R_F)$, so it follows that $\phi$ is an isomorphism. The inverse map $\phi^{-1}$ is well defined since for any two representative elements $u$, $v \in [w]_{R_F} \cap V(X)$ we have $[u]_R = [v]_R$.
\end{proof}

\begin{notation}
Let $R_{V(\sX)}$ be an equivalence relation on $V(\sX)$. For $\X \subseteq \sX$ we denote by $\X/R_{V(\sX)} := \{\, X//R_{V(\sX)} \mid X \in \X \,\}$ the corresponding collection of vertex-augmented quotient hypergraphs under $R_{V(\sX)}$.
\end{notation}

\section{Hypergraph transformations}\label{sec:Transformations}

\subsection{Definition and basic properties}\label{subsec:Transformations_def}

Our definition of a hypergraph transformation requires the following notion of maximality.

\begin{definition}[\textbf{Component-maximal set of hypergraphs}]\label{def:CompMax}
Suppose $\SC \subseteq  \X \subseteq \sX$. We say that $\SC$ is \emph{component maximal} in $\X$ if for each $X \in \X$ there exists a subset $\D_X \subseteq \SC$, called an \emph{$\SC$-maximal subset}, such that:
\begin{enumerate}[label=(\arabic*)]
\item $\D_X$ is pairwise component disjoint. \label{def:CompMax1}
\item If $\N \in \SC$ then $\N \in \D_X$. \label{def:CompMax2}
\item $\C(T) \subseteq \C(X)$ for all $T \in \D_X$. \label{def:CompMax3}
\item If $S \in \SC$ and $\C(S) \subseteq \C(X)$ then there exists $T \in \D_X$ such that $\C(S) \subseteq \C(T)$. \label{def:CompMax4}
\end{enumerate}
\end{definition}

\begin{proposition}\label{prop:CompMax}
Suppose $\SC \subseteq  \X \subseteq \sX$ where $\SC$ is component maximal in $\X$ with $\SC$-maximal subsets $\{\D_X\}_{X \in \X}$. Then:
\begin{enumerate}[label=(\arabic*)]
\item $\D_X$ is unique for all $X \in \X$. \label{prop:CompMax1}
\item If $\N \notin \SC$ then $\D_S = \{S\}$ for all $S \in \SC$. \label{prop:CompMax2}
\item If $\N \in \SC$ then $\D_S = \{\N,S\}$ for all $S \in \SC$, in particular $\D_{\N} = \{\N\}$. \label{prop:CompMax3}
\item If $\N \in \SC$ and $X \in \X$ then $\D_X = \{\N\}$ if and only if $\C(S) \nsubseteq \C(X)$ for all $S \in \SC \setminus \{\N\}$. \label{prop:CompMax4}
\item If $X$, $Y \in \X$, and $\C(S) \subseteq \C(X)$ if and only if $\C(S) \subseteq \C(Y)$ for all $S \in \SC \setminus \{\N\}$, then $\D_X = \D_Y$. \label{prop:CompMax5}
\item If $\N \notin \SC$ then $\D_X = \emptyset$ implies $\C(S) \nsubseteq \C(X)$ for all $S \in \SC$. \label{prop:CompMax6}
\end{enumerate}
\end{proposition}

\begin{proof}
\ref{prop:CompMax1} Fix $X \in \X$ and suppose that $\D'_X \subseteq \SC$ satisfies Properties~\ref{def:CompMax1} to \ref{def:CompMax4} in Definition~\ref{def:CompMax}. Since $\N \in \D'_X$ if and only if $\N \in \D_X$, it suffices to show that $\D'_X \setminus \{\N\} = \D_X \setminus \{\N\}$. If $S \in \D'_X \setminus \{\N\}$ then $\C(S) \subseteq \C(X)$, hence there exists $T \in \D_X \setminus \{\N\}$ such that $\C(S) \subseteq \C(T)$. Since $\C(T) \subseteq \C(X)$ there exists $S' \in \D'_X \setminus \{\N\}$ such that $\C(T) \subseteq \C(S')$. Then $\C(S) \subseteq \C(S')$ and the pairwise component disjointness of $\D'_X$ imply $S = S'$, hence $S = T \in \D_X \setminus \{\N\}$. We conclude that $\D'_X \setminus \{\N\} \subseteq \D_X \setminus \{\N\}$, where we may have $\D'_X \setminus \{\N\} = \emptyset$. An analogous argument shows that $\D_X \setminus \{\N\} \subseteq \D'_X \setminus \{\N\}$.

\ref{prop:CompMax2} If $\N \notin \SC$ and $S \in \SC$ then $\D := \{S\}$ satisfies Properties~\ref{def:CompMax1} to \ref{def:CompMax4} in Definition~\ref{def:CompMax} with respect to $S$. So, by Part~\ref{prop:CompMax1} of this proposition, we have $\D_S = \{S\}$.

\ref{prop:CompMax3} If $\N \in \SC$ and $S \in \SC \setminus \{\N\}$ then $\D := \{\N,S\}$ satisfies Properties~\ref{def:CompMax1} to \ref{def:CompMax4} in Definition~\ref{def:CompMax} with respect to $S$. So, by Part~\ref{prop:CompMax1} of this proposition, we have $\D_S = \{\N,S\}$. Additionally, if $T \in \D_{\N}$ then $\C(T) \subseteq \C(\N)$, hence $T = \N$, therefore $\D_{\N} = \{\N\}$.

\ref{prop:CompMax4} Let $\N \in \SC$ and $X \in \X$. For the forward direction, suppose $\D_X = \{\N\}$. If $S \in \SC$ and $\C(S) \subseteq \C(X)$ then Property~\ref{def:CompMax4} in Definition~\ref{def:CompMax} implies $\C(S) \subseteq \C(\N)$, hence $S = \N$. Therefore $\C(S) \nsubseteq \C(X)$ for all $S \in \SC \setminus \{\N\}$. For the reverse direction, suppose $\C(S) \nsubseteq \C(X)$ for all $S \in \SC \setminus \{\N\}$. Then $T \in \D_X$ implies $\C(T) \subseteq \C(X)$, by Property~\ref{def:CompMax3} in Definition~\ref{def:CompMax}, so $T = \N$ and hence $\D_X = \{\N\}$.

\ref{prop:CompMax5} Since $\N \in \D_X$ if and only if $\N \in \D_Y$, it suffices to show that $\D_X \setminus \{\N\} = \D_Y \setminus \{\N\}$. If $S \in \D_X \setminus \{\N\}$ then $\C(S) \subseteq \C(X)$, hence $\C(S) \subseteq \C(Y)$. So by Property~\ref{def:CompMax4} in Definition~\ref{def:CompMax} there exists $T \in \D_Y \setminus \{\N\}$ such that $\C(S) \subseteq \C(T) \subseteq \C(Y)$. Then $\C(T) \subseteq \C(X)$ so by Property~\ref{def:CompMax4} in Definition~\ref{def:CompMax} there exists $S' \in \D_X \setminus \{\N\}$ such that $\C(T) \subseteq \C(S') \subseteq \C(X)$. Since $\C(S) \subseteq \C(S')$, the pairwise component-disjointness of $\D_X$ implies $S = S'$ and hence $S = T \in \D_Y \setminus \{\N\}$. Thus $\D_X \setminus \{\N\} \subseteq \D_Y \setminus \{\N\}$. An analogous argument gives $\D_Y \setminus \{\N\} \subseteq \D_X \setminus \{\N\}$.

\ref{prop:CompMax6} We prove the contrapositive, so suppose there exists $S \in \SC$ such that $\C(S) \subseteq \C(X)$. Then by Property~\ref{def:CompMax4} in Definition~\ref{def:CompMax} there exists $T \in \D_X$ such that $\C(S) \subseteq \C(T)$, hence $\D_X \ne \emptyset$.
\end{proof}

Our definition of a \emph{hypergraph transformation} on $\sX$ ensures that the transformation acts consistently with respect to specified, or \emph{distinguished}, hypergraphs. We regard connected components as the fundamental units of hypergraphs, since a hypergraph can be decomposed into a direct sum of connected components, and modifying a particular connected component of a hypergraph has no effect on the relations described by any other connected component.

A \emph{partial transformation} on $\sX$ is a map $\pi \colon \X \to \sX$ where $\X \subseteq \sX$. The domain of $\pi$ is $\Dom(\pi) = \X$ and the image of $\pi$ is $\Image(\pi)$. A partial transformation therefore corresponds to a map between subsets of $\sX$.

\begin{definition}[\textbf{Hypergraph transformation}]\label{def:Ht}
A \emph{hypergraph transformation} on $\sX$ is a 3-tuple $\T := (\X,\pi,\SC)$ where $\pi \colon \X \to \sX$ is a partial transformation on $\sX$ and $\SC \subseteq \X$ is a collection of distinguished hypergraphs, satisfying each of the following conditions:
\begin{enumerate}[label={(\arabic*)}]
\item (Nonredundancy) $\C(S) \cap \C \big(\pi(S)\big) = \emptyset$ for all $S \in \SC$, and if $\N \in \SC$ then $\pi(\N) \ne \N$. \label{Ht1}
\item (Maximality) $\SC$ is component maximal in $\X$, with $\SC$-maximal subsets $\{\D_X\}_{X \in \X}$. \label{Ht2}
\item (Direct sum decomposition is preserved) For each $X \in \X$:
	\begin{enumerate}[label={(\alph*)}]
	\item Defining the set $\SC_X := \big\{\, S \in \D_X \mid V\big(\pi(S)\big) \cap V(X \ominus S) = \emptyset \,\big\}$, the set $\pi(\SC_X)$ consists of pairwise vertex-disjoint hypergraphs and $\lvert \pi(\SC_X) \rvert = \lvert \SC_X \rvert$.
	\item Denoting by $\widebar{X} \in \sX$ the induced strong subhypergraph of $X$ where $\widebar{X} := X \ominus (\bigoplus_{S \in \SC_X} S)$, the decomposition $X = \widebar{X} \oplus (\bigoplus_{S \in \SC_X} S)$ is preserved by $\pi$ to give $\pi(X) = \widebar{X} \oplus \big(\bigoplus_{S \in \SC_X} \pi(S)\big)$.
	\end{enumerate}\label{Ht3}
\end{enumerate}
\end{definition}

\begin{remarks}
With regard to Definition~\ref{def:Ht}:
\begin{enumerate}[label={(\arabic*)}]
\item Condition~\ref{Ht1} says that no connected component of a distinguished hypergraph $S \in \SC$ is fixed under $\pi$, ensuring that all components of $S$ are modified by $\pi$ and none are redundant. Further, if $\N \in \SC$ then $\N$ is not fixed under $\pi$, otherwise $\N$ is redundant as a distinguished hypergraph.
\item Condition~\ref{Ht3} specifies the direct sum decomposition of each hypergraph $X \in \X$ with respect to the distinguished hypergraphs, and the preservation of this direct sum decomposition under the action of $\pi$. In particular, $\pi$ is uniquely determined by $\SC$ and $\pi(\SC)$. Note that, since the hypergraphs in $\SC_X$ are pairwise component disjoint and are all subhypergraphs of the same hypergraph $X$, the set $\SC_X$ is pairwise vertex disjoint.
\item Employing partial transformations $\pi \colon \X \to \sX$ on $\sX$ provides flexibility for constructing hypergraph transformations, since we can choose a subset $\X$ on which $\pi$ has the appropriate action.
\item Hypergraph transformations are not in general closed under composition, and while we could weaken the defining conditions of hypergraph transformations to obtain closure under composition this would reduce the desired specificity of the transformations. Further, while our hypergraph transformations are unary functions they could readily be extended to $n$-ary functions for any $n \in \NN$.
\item The formal definition of a hypergraph transformation $\T := (\X,\pi,\SC)$ on $\sX$ is based on a collection of distinguished hypergraphs $\SC$, which ensures that the partial transformation $\pi \colon \X \to \sX$ acts consistently on $\X$. In practice, however, once we have established that $\T$ is a hypergraph transformation we can regard $\T$ as the partial transformation $\pi \colon \X \to \sX$ without further reference to $\SC$.
\end{enumerate}
\end{remarks}

We now discuss some properties of hypergraph transformations.

\begin{proposition}\label{prop:HTproperties}
Let $\T := (\X,\pi,\SC)$ be a hypergraph transformation on $\sX$.
\begin{enumerate}[label={(\arabic*)}]
\item $\SC_S = \{S\}$, for all $S \in \SC$. \label{prop:SS}
\item $\pi(X) = X$ if and only if $\SC_X = \emptyset$, for $X \in \X$. \label{prop:nonfixed}
\item $\SC = \emptyset$ if and only if $\pi(X) = X$ for all $X \in \X$, that is $\pi$ is an inclusion transformation. \label{prop:nonfixedcorr}
\end{enumerate}
\end{proposition}

\begin{proof}
\ref{prop:SS} Let $S \in \SC$. If $\N \notin \SC$ then $\D_S = \{S\}$, by Part~\ref{prop:CompMax2} of Proposition~\ref{prop:CompMax}, so $\SC_S = \{S\}$. Suppose now that $\N \in \SC$. Then $\D_S = \{\N,S\}$, by Part~\ref{prop:CompMax3} of Proposition~\ref{prop:CompMax}, so we must have $\SC_S = \{S\}$: we cannot have $\SC_S = \{\N,S\}$ when $S \ne \N$ since the direct sum decomposition gives $\pi(S) = \big(S \ominus (\bigoplus_{T \in \SC_S} T)\big) \oplus \big(\bigoplus_{T \in \SC_S} \pi(T)\big) = \big(S \ominus (\N \oplus S)\big) \oplus \big(\pi(\N) \oplus \pi(S)\big) = \pi(\N) \oplus \pi(S)$, hence $\pi(\N) = \N$, and therefore nonredundancy does not hold.

\ref{prop:nonfixed} For the forward direction we use a contrapositive argument, so suppose that $\SC_X \ne \emptyset$. First, suppose that there exists $T \in \SC_X$ such that $\pi(T) \ne \N$. Since $\C(T) \cap \C \big(\pi(T)\big) = \emptyset$ and $\C(T) \subseteq \C(X)$ it follows that $\C\big(\pi(T)\big) \cap \C(X) = \C\big(\pi(T)\big) \cap \C(X \ominus T)$, and since $V\big(\pi(T)\big) \cap V(X \ominus T) = \emptyset$ it follows that $\C\big(\pi(T)\big) \cap \C(X \ominus T) = \emptyset$, so $\C\big(\pi(T)\big) \cap \C(X) = \emptyset$. Further, since $\pi(X) = \widebar{X} \oplus \big(\bigoplus_{S \in \SC_X} \pi(S)\big)$ by Condition~\ref{Ht3} of Definition~\ref{def:Ht}, it follows that $\C\big(\pi(T)\big) \subseteq \C\big(\pi(X)\big)$. We conclude that $\C\big(\pi(X)\big) \ne \C(X)$, and hence $\pi(X) \ne X$. Second, suppose that $\pi(S) = \N$ for all $S \in \SC_X$, which implies $\lvert \SC_X \rvert = \lvert \pi(\SC_X) \rvert = \lvert \{\N\} \rvert = 1$. Then Condition~\ref{Ht3} of Definition~\ref{def:Ht} gives $\pi(X) = \widebar{X} = X \ominus \big(\bigoplus_{S \in \SC_X} S\big) \ne X$. For the backward direction, if $\SC_X = \emptyset$ then Condition~\ref{Ht3} of Definition~\ref{def:Ht} implies that $\pi(X) = \widebar{X} = X$.

\ref{prop:nonfixedcorr} First note that if $S \in \SC$ then $\SC_S = \{S\}$ by Part~\ref{prop:SS} of this proposition, so $\SC = \emptyset$ if and only if $\SC_X = \emptyset$ for all $X \in \X$. The result then follows from Part~\ref{prop:nonfixed} of this proposition.
\end{proof}

We define a notion of disjointness for hypergraph transformations, a property that ensures independence of action of the hypergraph transformations (see Proposition~\ref{prop:DisTransComm}, Corollary~\ref{cor:DisTransComm}, and Proposition~\ref{prop:DisTrans}).

\begin{definition}[\textbf{Disjoint hypergraph transformations}]\label{def:DisTrans}
Two hypergraph transformations $\T := (\X,\pi,\SC)$ and $\T' := (\X',\pi',\SC')$ on $\sX$ are \emph{disjoint} if for all $X \in \SC \cup \pi(\SC)$ and for all $Y \in \SC' \cup \pi'(\SC')$ the hypergraphs $X$ and $Y$ are vertex disjoint.
\end{definition}

\begin{notation}
For two partial transformations $\pi_1 \colon \mathcal{X}_1 \to \sX$ and $\pi_2 \colon \mathcal{X}_2 \to \sX$ on $\sX$ their composition $\pi_2 \circ \pi_1 \colon \Dom (\pi_2 \circ \pi_1) \to \sX$ is the partial transformation with domain $\Dom (\pi_2 \circ \pi_1) := \pi_1^{-1} \big(\Image (\pi_1) \cap \Dom (\pi_2) \big)$. Note that $\Dom (\pi_2 \circ \pi_1)$ is the largest possible domain for $\pi_2 \circ \pi_1$, and if $\Image (\pi_1) \cap \Dom (\pi_2) = \emptyset$ then $\Dom (\pi_2 \circ \pi_1) = \emptyset$ and $\pi_2 \circ \pi_1$ is the empty transformation.

More generally, if $(\pi_i \colon \mathcal{X}_i \to \sX)_{i=1}^n$ is a finite sequence of partial transformations on $\sX$, for some $n \in \NN$, then their composition $\bigcirc_{i=1}^n \pi_{n+1-i} := \pi_n \circ \cdots \circ \pi_2 \circ \pi_1$ is the partial transformation $\bigcirc_{i=1}^n \pi_{n+1-i} \colon \Dom (\bigcirc_{i=1}^n \pi_{n+1-i}) \to \sX$, noting $\Dom (\bigcirc_{i=1}^n \pi_{n+1-i})$ is the largest possible domain for $\bigcirc_{i=1}^n \pi_{n+1-i}$.

Denote by $S_n$ the set of all permutations of $[n]$. The \emph{coincidence set} $\Coin\big((\bigcirc_{i \in \sigma} \pi_{n+1-i})_{\sigma \in S_n}\big)$ of the sequence of all compositions $(\bigcirc_{i \in \sigma} \pi_{n+1-i})_{\sigma \in S_n}$ of $(\pi_i \colon \mathcal{X}_i \to \sX)_{i=1}^n$ is the maximum subset of the common domain $\bigcap_{\sigma \in S_n} \Dom (\bigcirc_{i \in \sigma} \pi_{n+1-i})$ such that for each hypergraph in $\Coin\big((\bigcirc_{i \in \sigma} \pi_{n+1-i})_{\sigma \in S_n}\big)$ the compositions $(\bigcirc_{i \in \sigma} \pi_{n+1-i})_{\sigma \in S_n}$ of $(\pi_i \colon \mathcal{X}_i \to \sX)_{i=1}^n$ have the same image.
\end{notation}

\begin{lemma}\label{lemma:DistHyper}
Let $\T := (\X,\pi,\SC)$ be a hypergraph transformation on $\sX$, and let $X \in \X$. Suppose $X'$, $X'' \in \sX$ are such that $\C(X') \subseteq \C(X)$, $V(X'') \cap V(X \ominus X') = \emptyset$, and $Y := (X \ominus X') \oplus X'' \in \X$. Suppose further that, for all $S \in \SC$, $S$ is component disjoint with both $X'$ and $X''$, and $\pi(S)$ is vertex disjoint with both $X'$ and $X''$. Then $\D_Y = \D_X$ and $\SC_Y = \SC_X$.
\end{lemma}

\begin{proof}
We first show that $\D_Y = \D_X$, which will follow from Part~\ref{prop:CompMax5} of Proposition~\ref{prop:CompMax} after establishing that $\C(S) \subseteq \C(X)$ if and only if $\C(S) \subseteq \C(Y)$ for all $S \in \SC \setminus \{\N\}$. Fix $S \in \SC \setminus \{\N\}$. If $\C(S) \subseteq \C(X)$ then, since $S \in \SC$ implies $\C(S) \cap \C(X') = \emptyset$, and since $\C(X') \subseteq \C(X)$, we have $\C(S) \subseteq \C(X \ominus X')$ and hence $\C(S) \subseteq \C(Y)$. Conversely, if $\C(S) \subseteq \C(Y)$ then, since $S \in \SC$ implies $\C(S) \cap \C(X'') = \emptyset$, we have $\C(S) \subseteq \C(X \ominus X') \subseteq \C(X)$. Since $S \in \SC \setminus \{\N\}$ is arbitrary, we conclude that $\D_Y = \D_X$.

To show that $\SC_Y = \SC_X$ we begin by establishing that if $S \in \SC$ satisfies $\C(S) \subseteq \C(X)$ and $\C(S) \subseteq \C(Y)$ then $V\big(\pi(S)\big) \cap V(X \ominus S) = V\big(\pi(S)\big) \cap V(Y \ominus S)$. Note that $Y \ominus S = \big((X \ominus X') \oplus X''\big) \ominus S = \big((X \ominus X') \ominus S \big) \oplus X'' = \big((X \ominus S) \ominus X' \big) \oplus X''$: the second equality follows from Part~\ref{prop:DirectDiff3} of Proposition~\ref{prop:DirectDiff} since $X''$ is disjoint with $X \ominus X'$, and $\C(S) \subseteq \C(X \ominus X')$ which also implies $X''$ is disjoint with $S$; and the third equality follows from Part~\ref{prop:DirectDiff2} of Proposition~\ref{prop:DirectDiff} since $X'$ is disjoint with $S$, $\C(S) \subseteq \C(X)$, and $\C(X') \subseteq \C(X)$. Now, since $S \in \SC$ implies $V\big(\pi(S)\big) \cap V(X') = \emptyset$ we have $V\big(\pi(S)\big) \cap V(X \ominus S) = V\big(\pi(S)\big) \cap V\big((X \ominus S) \ominus X'\big) \subseteq V\big(\pi(S)\big) \cap V(Y \ominus S)$, and since $S \in \SC$ implies $V\big(\pi(S)\big) \cap V(X'') = \emptyset$ we have $V\big(\pi(S)\big) \cap V(Y \ominus S) = V\big(\pi(S)\big) \cap V\big((X \ominus S) \ominus X'\big) \subseteq V\big(\pi(S)\big) \cap V(X \ominus S)$. Therefore $V\big(\pi(S)\big) \cap V(X \ominus S) = V\big(\pi(S)\big) \cap V(Y \ominus S)$.

We now show that $\SC_Y = \SC_X$. First, we have $\D_Y = \D_X$. Second, $S \in \D_Y = \D_X$ implies $\C(S) \subseteq \C(Y)$ and $\C(S) \subseteq \C(X)$, so $V\big(\pi(S)\big) \cap V(Y \ominus S) = V\big(\pi(S)\big) \cap V(X \ominus S)$. We conclude that $\SC_Y = \SC_X$.
\end{proof}

\begin{proposition}\label{prop:DisTransComm}
If two hypergraph transformations $\T := (\X,\pi,\SC)$ and $\T' := (\X',\pi',\SC')$ on $\sX$ are disjoint then $\Coin\big((\pi' \circ \pi, \pi \circ \pi' )\big) = \Dom (\pi' \circ \pi) \cap \Dom (\pi \circ \pi')$.
\end{proposition}

\begin{proof}
We have $\Coin\big((\pi' \circ \pi, \pi \circ \pi' )\big) \subseteq \Dom (\pi' \circ \pi) \cap \Dom (\pi \circ \pi')$ by the definition of the coincidence set, so to establish the reverse inclusion let $X \in \Dom (\pi' \circ \pi) \cap \Dom (\pi \circ \pi')$ and we show $\pi\big(\pi'(X)\big) = \pi'\big(\pi(X)\big)$.

Denoting $Y := \pi'(X) = \big(X \ominus (\bigoplus_{S' \in \SC'_X} S')\big) \oplus \big(\bigoplus_{S' \in \SC'_X} \pi'(S')\big)$, and noting that $Y \in \Dom(\pi)$ and Lemma~\ref{lemma:DistHyper} implies $\SC_Y = \SC_X$, we have
\begin{align*}
\pi\big(\pi'(X)\big) = \pi(Y) &= \big(Y \ominus (\textstyle\bigoplus_{S \in \SC_Y} S)\big) \oplus \big(\textstyle\bigoplus_{S \in \SC_Y} \pi(S)\big)\\
&= \Big[\Big(\big(X \ominus (\textstyle\bigoplus_{S' \in \SC'_X} S')\big) \oplus \big(\textstyle\bigoplus_{S' \in \SC'_X} \pi'(S')\big)\Big) \ominus (\textstyle\bigoplus_{S \in \SC_Y} S) \Big] \oplus \big(\textstyle\bigoplus_{S \in \SC_Y} \pi(S)\big)\\
&= \Big[\Big(\big(X \ominus (\textstyle\bigoplus_{S' \in \SC'_X} S')\big) \oplus \big(\textstyle\bigoplus_{S' \in \SC'_X} \pi'(S')\big)\Big) \ominus (\textstyle\bigoplus_{S \in \SC_X} S) \Big] \oplus \big(\textstyle\bigoplus_{S \in \SC_X} \pi(S)\big).
\end{align*}
Now, denoting $Z := \pi(X) = \big(X \ominus (\bigoplus_{S \in \SC_X} S)\big) \oplus \big(\bigoplus_{S \in \SC_X} \pi(S)\big)$, and noting that $Z \in \Dom(\pi')$ and Lemma~\ref{lemma:DistHyper} implies $\SC'_Z = \SC'_X$, we have
\begin{align*}
\pi'\big(\pi(X)\big) = \pi'(Z) &= \big(Z \ominus (\textstyle\bigoplus_{S' \in \SC'_Z} S')\big) \oplus \big(\textstyle\bigoplus_{S' \in \SC'_Z} \pi'(S')\big)\\
&= \Big[\Big(\big(X \ominus (\textstyle\bigoplus_{S \in \SC_X} S)\big) \oplus \big(\textstyle\bigoplus_{S \in \SC_X} \pi(S)\big)\Big) \ominus (\textstyle\bigoplus_{S' \in \SC'_Z} S') \Big] \oplus \big(\textstyle\bigoplus_{S' \in \SC'_Z} \pi'(S')\big)\\
&= \Big[\Big(\big(X \ominus (\textstyle\bigoplus_{S \in \SC_X} S)\big) \oplus \big(\textstyle\bigoplus_{S \in \SC_X} \pi(S)\big)\Big) \ominus (\textstyle\bigoplus_{S' \in \SC'_X} S') \Big] \oplus \big(\textstyle\bigoplus_{S' \in \SC'_X} \pi'(S')\big)\\
&= \Big[\Big(\big(X \ominus (\textstyle\bigoplus_{S \in \SC_X} S)\big) \ominus (\textstyle\bigoplus_{S' \in \SC'_X} S')\Big) \oplus \big(\textstyle\bigoplus_{S \in \SC_X} \pi(S)\big) \Big] \oplus \big(\textstyle\bigoplus_{S' \in \SC'_X} \pi'(S')\big)\\
&= \Big[\Big(\big(X \ominus (\textstyle\bigoplus_{S' \in \SC'_X} S')\big) \ominus (\textstyle\bigoplus_{S \in \SC_X} S)\Big) \oplus \big(\textstyle\bigoplus_{S \in \SC_X} \pi(S)\big) \Big] \oplus \big(\textstyle\bigoplus_{S' \in \SC'_X} \pi'(S')\big)\\
&= \Big[\Big(\big(X \ominus (\textstyle\bigoplus_{S' \in \SC'_X} S')\big) \ominus (\textstyle\bigoplus_{S \in \SC_X} S)\Big) \oplus \big(\textstyle\bigoplus_{S' \in \SC'_X} \pi'(S')\big) \Big] \oplus \big(\textstyle\bigoplus_{S \in \SC_X} \pi(S)\big)\\
&= \Big[\Big(\big(X \ominus (\textstyle\bigoplus_{S' \in \SC'_X} S')\big) \oplus \big(\textstyle\bigoplus_{S' \in \SC'_X} \pi'(S')\big) \Big) \ominus (\textstyle\bigoplus_{S \in \SC_X} S) \Big] \oplus \big(\textstyle\bigoplus_{S \in \SC_X} \pi(S)\big)\\
&= \pi\big(\pi'(X)\big),
\end{align*}
where the fifth and eighth equalities hold by the disjointness of $\T$ and $\T'$ and by Part~\ref{prop:DirectDiff3} of Proposition~\ref{prop:DirectDiff}, the sixth equality holds by the disjointness of $\T$ and $\T'$ and by Part~\ref{prop:DirectDiff2} of Proposition~\ref{prop:DirectDiff}, and the seventh equality holds since a direct sum of hypergraphs is independent of the order of the hypergraph summands.
\end{proof}

\begin{corollary}\label{cor:DisTransComm}
If the finite sequence of hypergraph transformations $\big(\T_i := (\X_i,\pi_i,\SC^i)\big)_{i=1}^n$ on $\sX$, for some $n \in \NN$ with $n \ge 2$, is pairwise disjoint then $\Coin\big((\bigcirc_{i \in \sigma} \pi_{n+1-i})_{\sigma \in S_n}\big) = \bigcap_{\sigma \in S_n} \Dom (\bigcirc_{i \in \sigma} \pi_{n+1-i})$.
\end{corollary}

\begin{proof}
Note that $\bigcap_{\sigma \in S_n} \Dom (\bigcirc_{i \in \sigma} \pi_{n+1-i}) \subseteq \Dom (\pi_j \circ \pi_k) \cap \Dom (\pi_k \circ \pi_j) = \Coin\big((\pi_j \circ \pi_k, \pi_k \circ \pi_j )\big)$ for all $j$, $k \in [n]$ with $j \ne k$, where the equality follows from Proposition~\ref{prop:DisTransComm}. Since any two permutations in $S_n$ can be transformed into each other by permuting adjacent elements, it therefore follows that $\bigcap_{\sigma \in S_n} \Dom (\bigcirc_{i \in \sigma} \pi_{n+1-i}) \subseteq \Coin\big((\bigcirc_{i \in \sigma} \pi_{n+1-i})_{\sigma \in S_n}\big)$.
\end{proof}

\begin{proposition}\label{prop:DisTrans}
Suppose that $\big(\T_i := (\X_i,\pi_i,\SC^i)\big)_{i=1}^m$ is a finite sequence of pairwise disjoint hypergraph transformations on $\sX$, for some $m \in \NN$, and let $X \in \Coin\big((\bigcirc_{i \in \sigma} \pi_{m+1-i})_{\sigma \in S_m}\big)$. Then we have the direct sum decompositions
\begin{equation}\label{eq:DisTrans1}
X = \widebar{X} \oplus (\textstyle\bigoplus_{i \in [m]} \textstyle\bigoplus_{S \in \SC_X^i} S)
\end{equation}
and
\begin{equation}\label{eq:DisTrans2}
\pi_m \circ \cdots \circ \pi_1 (X) = \widebar{X} \oplus \big(\textstyle\bigoplus_{i \in [m]} \textstyle\bigoplus_{S \in \SC_X^i} \pi_i(S)\big),
\end{equation}
where $\widebar{X}$ is the strong subhypergraph of $X$ determined by $\widebar{X} := X \ominus (\bigoplus_{i \in [m]} \bigoplus_{S \in \SC_X^i} S)$, and \eqref{eq:DisTrans2} is independent of the order of the $\pi_i$.
\end{proposition}

\begin{proof}
We establish the result by induction on $m$. The result holds for $m = 1$ by Condition~\ref{Ht3} of Definition~\ref{def:Ht}. Suppose now that the result holds for $m = n$, where $n \in \NN$, and we show the result for $m = n+1$. Define $\widebar{X} := X \ominus (\bigoplus_{i \in [n+1]} \bigoplus_{S \in \SC_X^i} S)$, from which we have $X = \widebar{X} \oplus (\bigoplus_{i \in [n+1]} \bigoplus_{S \in \SC_X^i} S)$, so Equation~\eqref{eq:DisTrans1} holds for $m=n+1$. Now, $X = \big(\widebar{X} \oplus (\bigoplus_{S \in \SC_X^{n+1}} S) \big) \oplus (\bigoplus_{i \in [n]} \bigoplus_{S \in \SC_X^i} S)$, since a direct sum of hypergraphs is independent of the order of the hypergraph summands, so $\widebar{X} \oplus (\bigoplus_{S \in \SC_X^{n+1}} S) = X \ominus (\bigoplus_{i \in [n]} \bigoplus_{S \in \SC_X^i} S)$. Since Equations~\eqref{eq:DisTrans1} and \eqref{eq:DisTrans2} hold for $m=n$, by assumption, we have $X = \big(\widebar{X} \oplus (\bigoplus_{S \in \SC_X^{n+1}} S)\big) \oplus (\bigoplus_{i \in [n]} \bigoplus_{S \in \SC_X^i} S)$ and $\pi_n \circ \cdots \circ \pi_1 (X) = \big(\widebar{X} \oplus (\bigoplus_{S \in \SC_X^{n+1}} S)\big) \oplus \big(\bigoplus_{i \in [n]} \bigoplus_{S \in \SC_X^i} \pi_i(S)\big) =: Y$. We need to determine $\pi_{n+1} (Y)$.

We can write $Y = (X \ominus X') \oplus X''$ for $X' := \bigoplus_{i \in [n]} \bigoplus_{S \in \SC_X^i} S$ and $X'' := \bigoplus_{i \in [n]} \bigoplus_{S \in \SC_X^i} \pi_i(S)$. Note that $\C(X') \subseteq \C(X)$. Further, since $\T_i$ for $i \in [n]$ are hypergraph transformations, and since $V(X \ominus X') \subseteq V(X \ominus S)$ for all $S \in \SC_X^i$ with $i \in [n]$, it follows that $V(X'') \cap V(X \ominus X') = \emptyset$. Since the hypergraph transformations $(\T_i)_{i=1}^{n+1}$ are pairwise disjoint it follows that, for all $S \in \SC^{n+1}$, $S$ is component disjoint with both $X'$ and $X''$, and $\pi_{n+1}(S)$ is vertex disjoint with both $X'$ and $X''$. Therefore $\SC^{n+1}_Y = \SC^{n+1}_X$ by Lemma~\ref{lemma:DistHyper}. Further, letting $\widebar{Y} := Y \ominus (\bigoplus_{S \in \SC_Y^{n+1}} S)$, we have $Y = \widebar{Y} \oplus (\bigoplus_{S \in \SC_Y^{n+1}} S)$ and $\pi_{n+1} (Y) = \widebar{Y} \oplus \big(\bigoplus_{S \in \SC_Y^{n+1}} \pi_{n+1} (S)\big)$. So,
\begin{align*}
\pi_{n+1} \circ \cdots \circ \pi_1 (X)
&= \pi_{n+1} (Y)
= \widebar{Y} \oplus \big(\textstyle\bigoplus_{S \in \SC_Y^{n+1}} \pi_{n+1} (S)\big)\\
&= \big(Y \ominus (\textstyle\bigoplus_{S \in \SC_Y^{n+1}} S)\big) \oplus \big(\textstyle\bigoplus_{S \in \SC_Y^{n+1}} \pi_{n+1} (S)\big)
= \big(Y \ominus (\textstyle\bigoplus_{S \in \SC_X^{n+1}} S)\big) \oplus \big(\textstyle\bigoplus_{S \in \SC_X^{n+1}} \pi_{n+1} (S)\big)\\
&= \Big[\Big(\big(\widebar{X} \oplus (\textstyle\bigoplus_{S \in \SC_X^{n+1}} S)\big) \oplus \big(\textstyle\bigoplus_{i \in [n]} \textstyle\bigoplus_{S \in \SC_X^i} \pi_i(S)\big)\Big) \ominus (\textstyle\bigoplus_{S \in \SC_X^{n+1}} S)\Big] \oplus \big(\textstyle\bigoplus_{S \in \SC_X^{n+1}} \pi_{n+1} (S)\big)\\
&= \Big[\Big(\big(\widebar{X} \oplus (\textstyle\bigoplus_{S \in \SC_X^{n+1}} S)\big) \ominus (\textstyle\bigoplus_{S \in \SC_X^{n+1}} S)\Big) \oplus \big(\textstyle\bigoplus_{i \in [n]} \textstyle\bigoplus_{S \in \SC_X^i} \pi_i(S)\big) \Big] \oplus \big(\textstyle\bigoplus_{S \in \SC_X^{n+1}} \pi_{n+1} (S)\big)\\
&= \Big(\widebar{X} \oplus \big(\textstyle\bigoplus_{i \in [n]} \textstyle\bigoplus_{S \in \SC_X^i} \pi_i(S)\big) \Big) \oplus \big(\textstyle\bigoplus_{S \in \SC_X^{n+1}} \pi_{n+1} (S)\big) = \widebar{X} \oplus \big(\textstyle\bigoplus_{i \in [n+1]} \textstyle\bigoplus_{S \in \SC_X^i} \pi_i(S)\big),
\end{align*}
where the sixth equality follows from the pairwise disjointness of the hypergraph transformations and by Part~\ref{prop:DirectDiff3} of Proposition~\ref{prop:DirectDiff}, the seventh equality follows from the pairwise disjointness of the hypergraph transformations and by Part~\ref{prop:DirectDiff1} of Proposition~\ref{prop:DirectDiff}, and the eighth equality holds since a direct sum of hypergraphs is independent of the order of the hypergraph summands. Therefore, Equation~\eqref{eq:DisTrans2} holds for $m=n+1$. Moreover, \eqref{eq:DisTrans2} is independent of the order of the $\pi_i$ since $X \in \Coin\big((\bigcirc_{i \in \sigma} \pi_{m+1-i})_{\sigma \in S_m}\big)$.
\end{proof}

Under appropriate circumstances we can modify the collection of distinguished hypergraphs of a given hypergraph transformation to obtain a new hypergraph transformation, and we now consider a particular class of such modifications.

\begin{definition}[\textbf{Upward closed subset of hypergraphs}]\label{def:UpCl}
Suppose $\SC \subseteq \X \subseteq \sX$. A subset $\SC' \subseteq \SC$ is \emph{upward closed} with respect to $\SC$ if $S \in \SC'$, $T \in \SC$, and $\C(S) \subseteq \C(T)$ imply $T \in \SC'$.
\end{definition}

\begin{proposition}\label{prop:UpCl}
Suppose $\SC' \subseteq \SC \subseteq \X \subseteq \sX$.
\begin{enumerate}[label=(\arabic*)]
\item $\SC'$ is upward closed with respect to $\SC$ if and only if $T \in \SC \setminus \SC'$, $S \in \SC$, and $\C(S) \subseteq \C(T)$ imply $S \in \SC \setminus \SC'$. \label{prop:UpCl1}
\item If $\SC'$ is upward closed with respect to $\SC$ and $\N \in \SC'$ then $\SC' = \SC$. \label{prop:UpCl2}
\item If $\SC'$ is upward closed with respect to $\SC$, and $\SC$ is component maximal in $\X$ with $\SC$-maximal subsets $\{\D_X\}_{X \in \X}$, then $\SC'$ is component maximal in $\X$ with $\SC'$-maximal subsets $\D'_X := \D_X \cap \SC'$ for $X \in \X$. \label{prop:UpCl3}
\end{enumerate}
\end{proposition}

\begin{proof}
\ref{prop:UpCl1} For the forward direction, suppose $\SC'$ is upward closed with respect to $\SC$, $T \in \SC \setminus \SC'$, $S \in \SC$, and $\C(S) \subseteq \C(T)$. Since $\SC'$ is upward closed and $T \notin \SC'$ we must have $S \notin \SC'$. For the reverse direction, if $S \in \SC'$, $T \in \SC$, and $\C(S) \subseteq \C(T)$ then we must have $T \in \SC'$. Hence $\SC'$ is upward closed with respect to $\SC$.

\ref{prop:UpCl2} For all $T \in \SC$ we have $\C(\N) \subseteq \C(T)$ and hence $T \in \SC'$. It follows that $\SC' = \SC$.

\ref{prop:UpCl3} We need to show that the subsets $\D'_X$, for $X \in \X$, satisfy Properties~\ref{def:CompMax1} to \ref{def:CompMax4} of Definition~\ref{def:CompMax}. If $\N \in \SC'$ then $\SC' = \SC$ by Part~\ref{prop:UpCl2} of this proposition and hence $\D'_X = \D_X$ for $X \in \X$, so $\SC'$ is component maximal. Suppose now that $\N \notin \SC'$, and let $X \in \X$. Properties~\ref{def:CompMax1} and \ref{def:CompMax3} follow immediately from the definition of $\D'_X$ in terms of $\D_X$, and Property~\ref{def:CompMax2} holds trivially since $\N \notin \SC'$. For Property~\ref{def:CompMax4}, suppose $S \in \SC'$ with $\C(S) \subseteq \C(X)$. Then, since $S \in \SC$, there exists $T \in \D_X$ such that $\C(S) \subseteq \C(T)$ by Property~\ref{def:CompMax4}. Now, since $\SC'$ is upward closed, it follows that $S \in \SC'$, $T \in \SC$, and $\C(S) \subseteq \C(T)$ imply $T \in \SC'$ and therefore $T \in \D'_X$. So $\SC'$ is component maximal in $\X$ with $\SC'$-maximal subsets $\D'_X$ for $X \in \X$.
\end{proof}

\begin{definition}[\textbf{Support, support reduction/augmentation}]\label{def:Supp}
Let $\T := (\X,\pi,\SC)$ be a hypergraph transformation on $\sX$. The \emph{support} of $\T$, denoted $\Supp(\T)$, is defined by $\Supp(\T) := \{\, X \in \X \mid \pi(X) \ne X \,\}$.

Let $\T' := (\X,\pi',\SC')$ be another hypergraph transformation on $\sX$. Then $\T'$ is a \emph{support reduction} of $\T$ corresponding to $\SC'$ if $\SC' \subseteq \SC$, $\SC'$ is upward closed with respect to $\SC$, and $\pi'|_{\SC'} = \pi|_{\SC'}$. In this case we also say that $\T$ is a \emph{support augmentation} of $\T'$ corresponding to $\SC$.
\end{definition}

\begin{remark}
For a hypergraph transformation $\T := (\X,\pi,\SC)$, the subset $\X \setminus \Supp(\T)$ of $\X$ is the set of fixed points of $\T$, that is $\X \setminus \Supp(\T) = \{\, X \in \X \mid \pi(X) = X \,\}$.
\end{remark}

\begin{lemma}\label{lemma:EqualHT}
If $\T := (\X,\pi,\SC)$ and $\T' := (\X,\pi',\SC)$ are two hypergraph transformations on $\sX$ and if $\pi'|_{\SC} = \pi|_{\SC}$ then $\pi' = \pi$, hence $\T = \T'$.
\end{lemma}

\begin{proof}
For notational clarity we denote $\T' := (\X,\pi',\SC')$, so that $\SC' = \SC$, and $\D'_X = \D_X$ for all $X \in \X$. Let $X \in \X$. Then $\SC'_X = \big\{\, S \in \D'_X \mid V\big(\pi'(S)\big) \cap V(X \ominus S) = \emptyset \,\big\} = \big\{\, S \in \D_X \mid V\big(\pi(S)\big) \cap V(X \ominus S) = \emptyset \,\big\} = \SC_X$. So defining $\widebar{X}' := X \ominus (\bigoplus_{S \in \SC'_X} S)$ and $\widebar{X} := X \ominus (\bigoplus_{S \in \SC_X} S)$ we have $\widebar{X}' = \widebar{X}$, therefore $\pi'(X) = \widebar{X}' \oplus \big(\bigoplus_{S \in \SC'_X} \pi'(S)\big) = \widebar{X} \oplus \big(\bigoplus_{S \in \SC_X} \pi(S)\big) = \pi(X)$. Therefore $\pi' = \pi$.
\end{proof}

\begin{proposition}\label{prop:HypTrUpCl}
Suppose $\T := (\X,\pi,\SC)$ is a hypergraph transformation on $\sX$, and $\SC' \subseteq \SC$ is an upward closed subset with respect to $\SC$. Then there exists a hypergraph transformation $\T' := (\X,\pi',\SC')$ such that $\T'$ is the unique support reduction of $\T$ corresponding to $\SC'$. Further, $\Supp(\T') \subseteq \Supp(\T)$, and for all $X \in \X$ we have $\D'_X = \D_X \cap \SC'$, $\SC'_X = \SC_X \cap \SC'$, and $\pi'(X) = \big(X \ominus (\bigoplus_{S \in \SC'_X} S)\big) \oplus \big(\bigoplus_{S \in \SC'_X} \pi(S)\big)$.
\end{proposition}

\begin{proof}
Since $\T$ is a hypergraph transformation, $\SC$ is component maximal with $\SC$-maximal subsets $\{\D_X\}_{X \in \X}$, and since $\SC'$ is upward closed with respect to $\SC$ it follows from Part~\ref{prop:UpCl3} of Proposition~\ref{prop:UpCl} that $\SC'$ is component maximal in $\X$ with $\SC'$-maximal subsets $\D'_X := \D_X \cap \SC'$ for $X \in \X$.

If $\N \in \SC'$ then, since $\SC'$ is upward closed with respect to $\SC$, Part~\ref{prop:UpCl2} of Proposition~\ref{prop:UpCl} implies $\SC' = \SC$. It follows from Lemma~\ref{lemma:EqualHT} that the only support reduction of $\T$ corresponding to $\SC'$ is $\T$ itself.

Suppose now that $\N \notin \SC'$. We construct the partial transformation $\pi' \colon \X \to \sX$ by defining $\pi'(X)$ for all $X \in \X$. First define $\pi^{\ast} \colon \SC' \to \sX$ by $\pi^{\ast}(S) := \pi(S)$ for all $S \in \SC'$. Let $X \in \X$ and define the set $\SC^{\ast}_X := \big\{\, S \in \D'_X \mid V\big(\pi^{\ast}(S)\big) \cap V(X \ominus S) = \emptyset \,\big\}$. Then $\SC^{\ast}_X = \SC_X \cap \SC'$: $S \in \SC^{\ast}_X$ if and only if $S \in \D'_X$ and $V\big(\pi^{\ast}(S)\big) \cap V(X \ominus S) = \emptyset$ if and only if $S \in \D_X \cap \SC'$ and $V\big(\pi(S)\big) \cap V(X \ominus S) = \emptyset$ if and only if $S \in \SC_X \cap \SC'$. Further, since $\pi(\SC_X)$ consists of pairwise vertex-disjoint hypergraphs the subset $\pi^{\ast}(\SC^{\ast}_X) = \pi(\SC^{\ast}_X) \subseteq \pi(\SC_X)$ is also pairwise vertex disjoint, and since $\lvert \pi(\SC_X) \rvert = \lvert \SC_X \rvert$ we have $\lvert \pi^{\ast}(\SC^{\ast}_X) \rvert = \lvert \pi(\SC^{\ast}_X) \rvert = \lvert \SC^{\ast}_X \rvert$. Define $\widebar{X}^{\ast} := X \ominus (\bigoplus_{S \in \SC^{\ast}_X} S)$ and $\pi'(X) := \widebar{X}^{\ast} \oplus \big(\bigoplus_{S \in \SC^{\ast}_X} \pi^{\ast}(S)\big)$.

We show that $\T' := (\X,\pi',\SC')$ is a hypergraph transformation. Note that $\pi'|_{\SC'} = \pi^{\ast} = \pi|_{\SC'}$: the first equality holds since if $T \in \SC'$ then $\D_T = \{T\}$ or $\D_T = \{\N,T\}$ by Parts~\ref{prop:CompMax2} and \ref{prop:CompMax3} of Proposition~\ref{prop:CompMax}, so $\D'_T = \{T\}$, hence $\SC^{\ast}_T = \{T\}$, therefore the definition of $\pi'$ gives $\pi'(T) = \pi^{\ast}(T)$; the second equality follows from the definition of $\pi^{\ast}$. For nonredundancy, if $S \in \SC'$ then $\C(S) \cap \C\big(\pi'(S)\big) = \C(S) \cap \C\big(\pi(S)\big) = \emptyset$, where the last equality holds since $\T$ is a hypergraph transformation. We have shown that $\SC'$ is component maximal in $\X$ with $\SC'$-maximal subsets $\D'_X := \D_X \cap \SC'$ for $X \in \X$. To show that the direct sum decomposition is preserved, let $X \in \X$. Define $\SC'_X := \big\{\, S \in \D'_X \mid V\big(\pi'(S)\big) \cap V(X \ominus S) = \emptyset \,\big\}$, and note that $\SC'_X = \SC^{\ast}_X$ since $\pi'|_{\SC'} = \pi^{\ast}$. Then $\pi'(\SC'_X) = \pi^{\ast}(\SC^{\ast}_X)$ consists of pairwise vertex-disjoint hypergraphs, and $\lvert \pi'(\SC'_X) \rvert = \lvert \pi^{\ast}(\SC^{\ast}_X) \rvert = \lvert \SC^{\ast}_X \rvert = \lvert \SC'_X \rvert$. Denoting $\widebar{X}' := X \ominus (\bigoplus_{S \in \SC'_X} S)$, and noting that $\widebar{X}' = \widebar{X}^{\ast}$ since $\SC'_X = \SC^{\ast}_X$, we have $\pi'(X) := \widebar{X}^{\ast} \oplus \big(\bigoplus_{S \in \SC^{\ast}_X} \pi^{\ast}(S)\big) = \widebar{X}' \oplus \big(\bigoplus_{S \in \SC'_X} \pi'(S)\big)$, so the decomposition $X = \widebar{X}' \oplus (\bigoplus_{S \in \SC'_X} S)$ is preserved by $\pi'$. We conclude that $\T'$ is a hypergraph transformation. Moreover, $\T'$ is a support reduction of $\T$ corresponding to $\SC'$, and uniqueness of $\T'$ follows from Lemma~\ref{lemma:EqualHT}.

Finally, for $X \in \X$, $\SC^{\ast}_X = \SC_X \cap \SC'$ and $\SC'_X = \SC^{\ast}_X$ imply $\SC'_X = \SC_X \cap \SC'$. Moreover, $\Supp(\T') \subseteq \Supp(\T)$: $X \in \Supp(\T')$ implies $\pi'(X) \ne X$ implies $\SC'_X \ne \emptyset$, by Part~\ref{prop:nonfixed} of Proposition~\ref{prop:HTproperties}, implies $\SC_X \ne \emptyset$, since $\SC'_X = \SC_X \cap \SC'$, implies $\pi(X) \ne X$, by Part~\ref{prop:nonfixed} of Proposition~\ref{prop:HTproperties}, implies $X \in \Supp(\T)$.
\end{proof}

\subsection{Examples of hypergraph transformations}\label{subsec:Transformations_Ex}
Here we discuss two main examples of hypergraph transformations, namely hyperedge addition/deletion and hypergraph addition/deletion, as well as a combined hypergraph\text{--}hyperedge addition transformation.

\subsubsection{Hyperedge addition/deletion hypergraph transformations}\label{subsubsec:Transformations_Ex_HE}
The \emph{hyperedge space} of $\sX$ generalises the notion of the edge space of a graph \cite[Chapter 1.9, Page 23]{Diestel2017}.

\begin{definition}[\textbf{Hyperedge space}]
The \emph{hyperedge space} $(\E, \boxplus, \boxdot, \ZZ /2 \ZZ)$ of $\sX$ is the vector space over the field $\ZZ /2 \ZZ$ with underlying set $\E := \powerset \big(E(\sX)\big)$, where addition $\boxplus$ is the operation of symmetric difference, and scalar multiplication $\boxdot$ is given by $0 \boxdot F := \emptyset$ and $1 \boxdot F := F$ for all $F \in \E$. The hyperedge space has a basis given by the collection of all singleton sets in $\E$.
\end{definition}

\begin{notation}
Suppose $H \subseteq E(\sX)$ is nonempty and $X \in \sX$ with $\bigcup H \subseteq V(X)$. We denote by $X \boxplus H$ the hypergraph in $\sX$ with vertex set $V(X)$ and hyperedge set $E(X) \boxplus H$.
\end{notation}

\begin{definition}[\textbf{Hyperedge addition/deletion partial transformation}]
Suppose $H \subseteq E(\sX)$ is nonempty and $\X \subseteq \sX$. We define the \emph{hyperedge addition/deletion partial transformation} $\pi_H \colon \X \to \sX$ such that, for $X \in \X$,
\begin{equation}
\pi_H (X) =
\begin{cases}
X \boxplus H & \text{if $\bigcup H \subseteq V(X)$},\\
X & \text{if $\bigcup H \nsubseteq V(X)$.}
\end{cases}
\end{equation}
Therefore, if $\bigcup H \subseteq V(X)$ then $\pi_H$ adds all of the hyperedges in $H \setminus E(X)$ to $X$ and deletes all of the hyperedges in $H \cap E(X)$ from $X$, otherwise $\pi_H$ fixes $X$.
\end{definition}

\begin{notation}
Suppose $H \subseteq E(\sX)$ and $X \in \sX$. We denote by $X \wedge H$ the hypergraph in $\sX$ given by $X \wedge H := \bigcup \{\, C \in \C(X) \mid \text{$e \cap V(C) \ne \emptyset$ for some $e \in H$} \,\}$. Note that $(X \wedge H) \wedge H = X \wedge H$, and if $H = \emptyset$ then $X \wedge H = \N$.
\end{notation}

\begin{definition}[\textbf{$H$-closed set of hypergraphs}]
Suppose $H \subseteq E(\sX)$ is nonempty and $\X \subseteq \sX$. We say that $\X$ is \emph{$H$-closed for addition} (resp. \emph{$H$-closed for deletion}) if $X \in \X$, $\bigcup H \subseteq V(X)$, and $H \cap E(X) = \emptyset$ (resp. $H \subseteq E(X)$) imply $X \wedge H \in \X$. We say that $\X$ is \emph{$H$-closed} if $X \in \X$ and $\bigcup H \subseteq V(X)$ imply $X \wedge H \in \X$.
\end{definition}

\begin{proposition}\label{prop:HEpt}
Suppose $H \subseteq E(\sX)$ is nonempty, $\X \subseteq \sX$ is $H$-closed, and $\pi_H \colon \X \to \sX$ is the hyperedge addition/deletion partial transformation. Define
\begin{equation}
\SC := \{\, S \in \X \mid \text{$\textstyle\bigcup H \subseteq V(S)$ and $S = S \wedge H$} \,\}, \label{prop:HEpt1}
\end{equation}
and for $X \in \X$ define
\begin{equation} \label{prop:HEt4}
\D_X :=
\begin{cases}
\{X \wedge H\} & \text{if $\bigcup H \subseteq V(X)$},\\
\emptyset & \text{if $\bigcup H \nsubseteq V(X)$}.
\end{cases}
\end{equation}
Then $\T_H := (\X,\pi_H,\SC)$ is a hypergraph transformation on $\sX$ with $\SC$-maximal subsets $\{\D_X\}_{X \in \X}$, and the collection of distinguished hypergraphs $\SC$ is greatest for $\pi_H$ with respect to inclusion.
\end{proposition}

\begin{proof}
For nonredundancy, if $S \in \SC$ then $\pi_H(S) = S \boxplus H$ modifies all components of $S$, since $S = S \wedge H$, so $\C(S) \cap \C\big(\pi_H(S)\big) = \emptyset$; additionally, $\N \notin \SC$.

To see that $\SC$ is component maximal, let $X \in \X$. Conditions~\ref{def:CompMax1} to \ref{def:CompMax3} of Definition~\ref{def:CompMax} follow immediately from the definition of $\D_X$. For Condition~\ref{def:CompMax4} of Definition~\ref{def:CompMax}, if $S \in \SC$ and $\C(S) \subseteq \C(X)$ then $S = X \wedge H$, so $\D_X = \{S\}$; note that if $\C(S) \nsubseteq \C(X)$ for all $S \in \SC$ then, since $\X$ is $H$-closed, $\D_X = \emptyset$. It follows that $\D_X$ is $\SC$-maximal.

To show the direct sum decomposition is preserved, let $X \in \X$. Note that $\SC_X = \D_X$, so $\lvert \SC_X \rvert = \lvert \pi_H(\SC_X) \rvert \le 1$, and it also holds trivially that $\pi_H(\SC_X)$ is pairwise vertex disjoint. Now, note that $\SC_X \ne \emptyset$ if and only if $\D_X \ne \emptyset$ if and only if $\bigcup H \subseteq V(X)$. So, if $\SC_X = \{S\}$ then $\pi_H (X) = X \boxplus H = (X \ominus S) \oplus (S \boxplus H) = \widebar{X} \oplus \pi_H (S)$ where $\widebar{X} = X \ominus S$; and if $\SC_X = \emptyset$ then $\pi_H (X) = X = \widebar{X}$. In any case we have $\pi_H (X) = \widebar{X} \oplus \big(\bigoplus_{S \in \SC_X} \pi_H(S)\big)$, where $\widebar{X} := X \ominus \big(\bigoplus_{S \in \SC_X} S\big)$. We conclude that $\T_H$ is a hypergraph transformation on $\sX$ with $\SC$-maximal subsets $\{\D_X\}_{X \in \X}$.

Finally we show that $\SC$ is greatest for $\pi_H$, so let $T \in \X$ be such that $\big(\X,\pi_H,\SC \cup \{T\}\big)$ is a hypergraph transformation. Nonredundancy implies $T \ne \N$ since $\pi_H (\N) = \N$, and also $\pi_H (T) \ne T$ since $\C(T) \cap \C\big(\pi_H (T)\big) = \emptyset$, hence we must have $\bigcup H \subseteq V(T)$ by the definition of $\pi_H$. Suppose there exists $C \in \C(T)$ such that $e \cap V(C) = \emptyset$ for all $e \in H$. Then $\C\big(\pi_H (T)\big) = \C(T \boxplus H) = \{C\} \cup \C\big((T \setminus C) \boxplus H\big)$, so $C \in \C(T) \cap \C\big(\pi_H (T)\big)$, contradicting nonredundancy. So we must have $T = T \wedge H$. Therefore $T \in \SC$.
\end{proof}

\begin{definition}[\textbf{Hyperedge addition partial transformation}]
Suppose $H \subseteq E(\sX)$ is nonempty and $\X \subseteq \sX$. We define the \emph{hyperedge addition partial transformation} $\pi^+_H \colon \X \to \sX$ such that, for $X \in \X$,
\begin{equation}
\pi^+_H (X) =
\begin{cases}
X \boxplus H & \text{if $\bigcup H \subseteq V(X)$ and $H \cap E(X) = \emptyset$},\\
X & \text{if $\bigcup H \nsubseteq V(X)$ or $H \cap E(X) \ne \emptyset$}.
\end{cases}
\end{equation}
Therefore, if $\bigcup H \subseteq V(X)$ and $H \cap E(X) = \emptyset$ then $\pi^+_H$ adds all of the hyperedges in $H$ to $X$, otherwise $\pi^+_H$ fixes $X$.
\end{definition}

\begin{proposition}\label{prop:HEtadd}
Suppose $H \subseteq E(\sX)$ is nonempty, $\X \subseteq \sX$ is $H$-closed for addition, and $\pi^+_H \colon \X \to \sX$ is the hyperedge addition partial transformation. Define
\begin{equation}
\SC^+ := \{\, S \in \X \mid \text{$\textstyle\bigcup H \subseteq V(S)$, $H \cap E(S) = \emptyset$, and $S = S \wedge H$} \,\}, \label{prop:HEtadd3}
\end{equation}
and for $X \in \X$ define
\begin{equation} \label{prop:HEtadd4}
\D^+_X :=
\begin{cases}
\{X \wedge H\} & \text{if $\bigcup H \subseteq V(X)$ and $H \cap E(X) = \emptyset$},\\
\emptyset & \text{if $\bigcup H \nsubseteq V(X)$ or $H \cap E(X) \ne \emptyset$}.
\end{cases}
\end{equation}
Then:
\begin{enumerate}[label={(\arabic*)}]
\item $\T^+_H := (\X,\pi^+_H,\SC^+)$ is a hypergraph transformation on $\sX$ with $\SC^+$-maximal subsets $\{\D^+_X\}_{X \in \X}$. \label{prop:HEtadd5}
\item If $\X$ is $H$-closed then $\T^+_H$ is the support reduction of the hypergraph addition/deletion transformation $\T_H := (\X,\pi_H,\SC)$ corresponding to $\SC^+$, such that $\SC^+ = \{\, S \in \SC \mid H \cap E(S) = \emptyset \,\}$ and $\D^+_X = \D_X \cap \SC^+$ for $X \in \X$. \label{prop:HEtadd6}
\end{enumerate}
\end{proposition}

\begin{proof}
\ref{prop:HEtadd5} We omit the proof as it is similar to the proof of Proposition~\ref{prop:HEpt}.

\ref{prop:HEtadd6} The subset $\SC^+ = \{\, S \in \SC \mid H \cap E(S) = \emptyset \,\}$ is upward closed with respect to $\SC$ since $S \in \SC^+$, $T \in \SC$, and $\C(S) \subseteq \C(T)$ imply $T = S$ and hence $T \in \SC^+$. Further, $\pi^+_H|_{\SC^+} = \pi_H|_{\SC^+}$, and Proposition~\ref{prop:HEpt} implies $\T_H$ is a hypergraph transformation since $\X$ is $H$-closed. It follows from Proposition~\ref{prop:HypTrUpCl} that $\T^+_H$ is the support reduction of $\T_H$ corresponding to $\SC^+$, and $\D^+_X = \D_X \cap \SC^+$ for $X \in \X$.
\end{proof}

\subsubsection{Hypergraph addition/deletion hypergraph transformations}\label{subsubsec:Transformations_Ex_HG}
The \emph{component space} of $\sX$ is a vector space of connected components.

\begin{definition}[\textbf{Component space}]
The \emph{component space} $(\Z, \boxplus, \boxdot, \ZZ /2 \ZZ)$ of $\sX$ is the vector space over the field $\ZZ /2 \ZZ$ with underlying set $\Z := \powerset \big(\C(\sX)\big)$, where addition $\boxplus$ is the operation of symmetric difference, and scalar multiplication $\boxdot$ is given by $0 \boxdot G := \emptyset$ and $1 \boxdot G := G$ for all $G \in \Z$. The component space has a basis given by the collection of all singleton sets in $\Z$.
\end{definition}

\begin{remark}
Note that while the notation for the addition and scalar multiplication operations is the same for both the hyperedge space and the component space, no ambiguity is possible.
\end{remark}

\begin{definition}[\textbf{Hypergraph addition/deletion partial transformation}]\label{def:HadT}
Let $W \in \sX^{\ast}$, and $\X \subseteq \sX$ with $\N \in \X$. We define the \emph{hypergraph addition/deletion partial transformation} $\pi_W \colon \X \to \sX$ such that, for $X \in \X$,
\begin{equation}
\pi_W (X) =
\begin{cases}
\bigoplus \C(X) \boxplus \C(W) & \text{if $\C(W) \subseteq \C(X)$ or $V(W) \cap V(X) = \emptyset$},\\
X & \text{if $\C(W) \nsubseteq \C(X)$ and $V(W) \cap V(X) \neq \emptyset$}.
\end{cases}
\end{equation}
Therefore, if $V(W) \cap V(X) = \emptyset$ then $\pi_W$ adds the components of the hypergraph $W$ to $X$, if $\C(W) \subseteq \C(X)$ then $\pi_W$ deletes the components of the hypergraph $W$ from $X$, otherwise $\pi_W$ fixes $X$.
\end{definition}

\begin{proposition}\label{prop:HadT}
Suppose $W \in \sX^{\ast}$, $\X \subseteq \sX$ with $\N \in \X$, and $\pi_W \colon \X \to \sX$ is the hypergraph addition/deletion partial transformation. Define
\begin{equation} \label{prop:HadT2}
\SC := \{ \N, W \},
\end{equation}
and for $X \in \X$ define
\begin{equation} \label{prop:HadT3}
\D_X :=
\begin{cases}
\{\N, W\} & \text{if $\C(W) \subseteq \C(X)$},\\
\{\N\} & \text{if $\C(W) \nsubseteq \C(X)$}.
\end{cases}
\end{equation}
Then $\T_W := (\X,\pi_W,\SC)$ is a hypergraph transformation on $\sX$ with $\SC$-maximal subsets $\{\D_X\}_{X \in \X}$, and the collection of distinguished hypergraphs $\SC$ is greatest for $\pi_W$ with respect to inclusion.
\end{proposition}

\begin{proof}
For nonredundancy, let $S \in \SC$. If $S = \N$ then $\C(S) = \emptyset$, hence $\C(S) \cap \C\big(\pi_W (S)\big) = \emptyset$. If $S = W$ then $\pi_W (S) = \bigoplus \C(S) \boxplus \C(W) = \bigoplus \C(W) \boxplus \C(W) = \bigoplus \emptyset = \N$, so $\C(S) \cap \C\big(\pi_W (S)\big) = \emptyset$. Further, $\N \in \SC$ and $\pi_W (\N) = \bigoplus \C(\N) \boxplus \C(W) = \bigoplus \C(W) = W \ne \N$.

The set $\SC$ is component maximal since Conditions~\ref{def:CompMax1} to \ref{def:CompMax4} of Definition~\ref{def:CompMax} follow immediately from the definition of the subsets $\{\D_X\}_{X \in \X}$.

To show the direct sum decomposition is preserved, let $X \in \X$. Then $\SC_X$ takes one of the following forms: if $\C(W) \nsubseteq \C(X)$ and $V(W) \cap V(X) \neq \emptyset$ then $\D_X = \{\N\}$ and $\SC_X = \emptyset$; if $\C(W) \nsubseteq \C(X)$ and $V(W) \cap V(X) = \emptyset$ then $\D_X = \{\N\}$ and $\SC_X = \{\N\}$; and, if $\C(W) \subseteq \C(X)$ then $\D_X = \{ \N, W \}$, and also $V(W) \cap V(X) \neq \emptyset$, so $\SC_X = \{W\}$. Then $\lvert \SC_X \rvert = \lvert \pi_W(\SC_X) \rvert \le 1$, and it also holds trivially that $\pi_W(\SC_X)$ is pairwise vertex disjoint. We consider three cases. First, if $\C(W) \nsubseteq \C(X)$ and $V(W) \cap V(X) \neq \emptyset$ then $\SC_X = \emptyset$ and $\pi_W (X) = X = \widebar{X} = \widebar{X} \oplus \big(\bigoplus_{S \in \SC_X} \pi_W (S)\big)$ where $\widebar{X} := X \ominus \big(\bigoplus_{S \in \SC_X} S\big)$. Second, if $\C(W) \nsubseteq \C(X)$ and $V(W) \cap V(X) = \emptyset$ then $\SC_X = \{\N\}$ and $\pi_W (X) = \bigoplus \C(X) \boxplus \C(W) = \bigoplus \C(X) \oplus \bigoplus \C(W) = X \oplus W = (X \ominus \N) \oplus \pi_W (\N) = \big(X \ominus (\bigoplus_{S \in \SC_X} S)\big) \oplus \big(\bigoplus_{S \in \SC_X} \pi_W (S)\big) = \widebar{X} \oplus \big(\bigoplus_{S \in \SC_X} \pi_W (S)\big)$, where $\widebar{X} := X \ominus \big(\bigoplus_{S \in \SC_X} S\big)$. Third, if $\C(W) \subseteq \C(X)$ then $\SC_X = \{W\}$ and $\pi_W (X) = \bigoplus \C(X) \boxplus \C(W) = \bigoplus \C(X) \ominus \bigoplus \C(W) = X \ominus W = \big(X \ominus W\big) \oplus \N = \big(X \ominus W\big) \oplus \pi_W (W) = \big(X \ominus (\bigoplus_{S \in \SC_X} S)\big) \oplus \big(\bigoplus_{S \in \SC_X} \pi_W (S)\big) = \widebar{X} \oplus \big(\bigoplus_{S \in \SC_X} \pi_W (S)\big)$, where $\widebar{X} := \big(X \ominus (\bigoplus_{S \in \SC_X} S)\big)$.

Finally we show that $\SC$ is greatest for $\pi_W$, so let $T \in \X$ be such that $\big(\X,\pi_W,\SC \cup \{T\}\big)$ is a hypergraph transformation. Note that either $\C(W) \subseteq \C(T)$ or $V(W) \cap V(T) = \emptyset$, otherwise $\pi_W (T) = T$ which contradicts nonredundancy. First, suppose $\C(W) \subseteq \C(T)$. If $T \ne W$ then there exists $C \in \C(T) \setminus \C(W)$, and since $\pi_W (T) = \bigoplus \C(T) \boxplus \C(W) = \bigoplus \C(T) \setminus \C(W)$ we have $C \in \C(T) \cap \C(\pi_W(T))$, contradicting nonredundancy, hence $T = W \in \SC$. Second, suppose $V(W) \cap V(T) = \emptyset$. Then $\C\big(\pi_W (T)\big) = \C\big(\bigoplus \C(T) \boxplus \C(W)\big) = \C(T) \cup \C(W)$, and by nonredundancy we have $\emptyset = \C(T) \cap \C\big(\pi_W (T)\big) = \C(T)$, therefore $T = \N \in \SC$.
\end{proof}

\begin{definition}[\textbf{Hypergraph addition partial transformation}]\label{def:HaT}
Let $W \in \sX^{\ast}$, and $\X \subseteq \sX$ with $\N \in \X$. We define the \emph{hypergraph addition partial transformation} $\pi^+_W \colon \X \to \sX$ such that, for $X \in \X$,
\begin{equation}
\pi^+_W (X) =
\begin{cases}
X \oplus W & \text{if $V(W) \cap V(X) = \emptyset$},\\
X & \text{if $V(W) \cap V(X) \neq \emptyset$}.
\end{cases}
\end{equation}
Therefore, if $V(W) \cap V(X) = \emptyset$ then $\pi^+_W$ adds the components of the hypergraph $W$ to $X$, otherwise $\pi^+_W$ fixes $X$.
\end{definition}

\begin{proposition}\label{prop:HaT}
Suppose $W \in \sX^{\ast}$, $\X \subseteq \sX$ with $\N \in \X$, and $\pi^+_W \colon \X \to \sX$ is the hypergraph addition partial transformation. Define
\begin{equation}
\SC^+ := \{\N\}, \label{prop:HaT2}
\end{equation}
and for $X \in \X$ define
\begin{equation}
\D^+_X := \{\N\}. \label{prop:HaT3}
\end{equation}
Then $\T^+_W := (\X,\pi^+_W,\SC^+)$ is a hypergraph transformation on $\sX$ with $\SC^+$-maximal subsets $\{\D^+_X\}_{X \in \X}$.
\end{proposition}

\begin{proof}
We omit the proof as it is similar to the proof of Proposition~\ref{prop:HadT}.
\end{proof}
%
%
%

\begin{remark}
Suppose $W \in \sX^{\ast}$, and $\X \subseteq \sX$ with $\N \in \X$. Then $\SC^+$ in Equation~\eqref{prop:HaT2} of Proposition~\ref{prop:HaT} is not upward closed with respect to $\SC$ in Equation~\eqref{prop:HadT2} of Proposition~\ref{prop:HadT}, since $\N \in \SC^+$, $W \in \SC$, and $\C(\N) \subseteq \C(W)$, however $W \notin \SC^+$. In fact, if $\SC^+$ was upward closed with respect to $\SC$ then Part~\ref{prop:UpCl2} of Proposition~\ref{prop:UpCl} would imply that $\SC^+ = \SC$. In particular, $\T^+_W$ is not a support reduction of $\T_W$.
\end{remark}

\subsubsection{Combined hypergraph--hyperedge addition hypergraph transformations}\label{subsubsec:Transformations_Ex_HGHE}
Here we provide an example of a hypergraph transformation that performs a hypergraph addition followed by the addition of hyperedges.

\begin{definition}[\textbf{Hypergraph--hyperedge addition partial transformation}]\label{def:HGHEaddT}
Suppose $H \subseteq E(\sX)$ is nonempty, $W \in \sX^{\ast}$, and $\X \subseteq \sX$ with $V(\X) \cap V(W) = \emptyset$. We define the \emph{hypergraph--hyperedge addition partial transformation} $\pi^{++}_{W,H} \colon \X \to \sX$ such that, for $X \in \X$,
\begin{equation}
\pi^{++}_{W,H} (X) =
\begin{cases}
(X \oplus W) \boxplus H & \text{if $\bigcup H \subseteq V(X \oplus W)$ and $H \cap E(X \oplus W) = \emptyset$},\\
X & \text{if $\bigcup H \nsubseteq V(X \oplus W)$ or $H \cap E(X \oplus W) \ne \emptyset$}.
\end{cases}
\end{equation}
Therefore, if $\bigcup H \subseteq V(X \oplus W)$ and $H \cap E(X \oplus W) = \emptyset$ then $\pi^{++}_{W,H}$ adds the components of the hypergraph $W$ to $X$ and then adds all of the hyperedges in $H$ to $X \oplus W$, otherwise $\pi^{++}_{W,H}$ fixes $X$.
\end{definition}

\begin{definition}[\textbf{$W\text{--}H$-closed for addition set of hypergraphs}]
Suppose $H \subseteq E(\sX)$ is nonempty, $W \in \sX^{\ast}$, and $\X \subseteq \sX$ with $V(\X) \cap V(W) = \emptyset$. We say that $\X$ is \emph{$W\text{--}H$-closed for addition} if $X \in \X$, $\bigcup H \subseteq V(X \oplus W)$, $H \cap E(X \oplus W) = \emptyset$, and $X \wedge H \ne \N$ imply $X \wedge H \in \X$.
\end{definition}

\begin{proposition}\label{prop:HGHEtadd}
Suppose $H \subseteq E(\sX)$ is nonempty, $W \in \sX^{\ast}$, $\X \subseteq \sX$ with $V(\X) \cap V(W) = \emptyset$ is $W\text{--}H$-closed for addition, and $\pi^{++}_{W,H} \colon \X \to \sX$ is the hypergraph--hyperedge addition partial transformation. Define
\begin{equation}
\SC^{++} := \{\, X \wedge H \mid \text{$X \in \X$, $\textstyle\bigcup H \subseteq V(X \oplus W)$, $H \cap E(X \oplus W) = \emptyset$, and $X \wedge H \ne \N$} \,\}, \label{prop:HGHEtadd1}
\end{equation}
and for $X \in \X$ define
\begin{equation} \label{prop:HGHEtadd2}
\D^{++}_X :=
\begin{cases}
\{X \wedge H\} & \text{if $\bigcup H \subseteq V(X \oplus W)$, $H \cap E(X \oplus W) = \emptyset$, and $X \wedge H \ne \N$},\\
\emptyset & \text{if $\bigcup H \nsubseteq V(X \oplus W)$ or $H \cap E(X \oplus W) \ne \emptyset$ or $X \wedge H = \N$}.
\end{cases}
\end{equation}
Then $\T^{++}_{W,H} := (\X,\pi^{++}_{W,H},\SC^{++})$ is a hypergraph transformation on $\sX$ with $\SC^{++}$-maximal subsets $\{\D^{++}_X\}_{X \in \X}$.
\end{proposition}

\begin{proof}
For nonredundancy, suppose $S := X \wedge H \in \SC^{++}$ for some $X \in \X$. Then $\textstyle\bigcup H \subseteq V(X \oplus W)$ implies $\textstyle\bigcup H \subseteq V(S \oplus W)$, and $H \cap E(X \oplus W) = \emptyset$ implies $H \cap E(S \oplus W) = \emptyset$. So $\pi^{++}_{W,H}(S) = (S \oplus W) \boxplus H$ modifies all components of $S$, since $S = S \wedge H$ and $S \ne \N$, and it follows that $\C(S) \cap \C\big(\pi^{++}_{W,H}(S)\big) = \emptyset$, noting $V(W) \cap V(S) = \emptyset$. Additionally, $\N \notin \SC^{++}$.

To see that $\SC^{++}$ is component maximal, let $X \in \X$. Conditions~\ref{def:CompMax1} to \ref{def:CompMax3} of Definition~\ref{def:CompMax} follow immediately from the definition of $\D^{++}_X$. For Condition~\ref{def:CompMax4} of Definition~\ref{def:CompMax}, suppose $S := X \wedge H \in \SC^{++}$ for some $X \in \X$ and $\C(S) \subseteq \C(Y)$ for some $Y \in \X$. Then: $\C(S) \subseteq \C(Y \wedge H)$; $S \ne \N$ implies $Y \wedge H \ne \N$; $\bigcup H \subseteq V(X \oplus W)$ implies $\bigcup H \subseteq V(S \oplus W)$, hence $\bigcup H \subseteq V(Y \oplus W)$; $\C(S \oplus W) \subseteq \C(Y \oplus W)$ and $\bigcup H \subseteq V(S \oplus W) \subseteq V(Y \oplus W)$ imply $(S \oplus W) \wedge H = (Y \oplus W) \wedge H$, and since $H \cap E(S \oplus W) = H \cap E(X \oplus W) = \emptyset$ we have $H \cap E(Y \oplus W) = \emptyset$. So $\D^{++}_Y = \{Y \wedge H\}$ is the required $\SC^{++}$-maximal subset. Note that if $\C(S) \nsubseteq \C(Y)$ for all $S \in \SC^{++}$ then, since $\X$ is $W\text{--}H$-closed for addition, $\D^{++}_Y = \emptyset$.

To show the direct sum decomposition is preserved, let $X \in \X$. Note that $\SC^{++}_X = \D^{++}_X$, so $\lvert \SC^{++}_X \rvert = \lvert \pi^{++}_{W,H}(\SC^{++}_X) \rvert \le 1$, and it also holds trivially that $\pi^{++}_{W,H}(\SC^{++}_X)$ is pairwise vertex disjoint. Now, note that $\SC^{++}_X \ne \emptyset$ if and only if $\D^{++}_X \ne \emptyset$ if and only if $\bigcup H \subseteq V(X \oplus W)$, $H \cap E(X \oplus W) = \emptyset$, and $X \wedge H \ne \N$. So, if $\SC^{++}_X = \{S\}$ then $\pi^{++}_{W,H}(X) = (X \oplus W) \boxplus H = (X \ominus S) \oplus \big((S \oplus W) \boxplus H\big) = \widebar{X} \oplus \pi^{++}_{W,H} (S)$ where $\widebar{X} = X \ominus S$; and if $\SC^{++}_X = \emptyset$ then $\pi^{++}_{W,H} (X) = X = \widebar{X}$. In any case we have $\pi^{++}_{W,H} (X) = \widebar{X} \oplus \big(\bigoplus_{S \in \SC^{++}_X} \pi^{++}_{W,H}(S)\big)$, where $\widebar{X} := X \ominus \big(\bigoplus_{S \in \SC^{++}_X} S\big)$. We conclude that $\T^{++}_{W,H}$ is a hypergraph transformation on $\sX$ with $\SC^{++}$-maximal subsets $\{\D^{++}_X\}_{X \in \X}$.
\end{proof}

\section{Quotient hypergraph transformations}\label{sec:Quotients}

\subsection{Definition and basic properties}\label{subsec:Quotients_def}
Given a hypergraph transformation $\T$ on $\sX$ and an equivalence relation $R_{V(\sX)}$ on the set of vertices $V(\sX)$, we consider the notion of a corresponding \emph{quotient hypergraph transformation} of $\T$ on the hypergraph family $\sX/R_{V(\sX)}$. The existence of the quotient of a hypergraph transformation depends on the particular equivalence relation $R_{V(\sX)}$, since hypergraphs in $\sX$ associated with $\T$ must project into $\sX/R_{V(\sX)}$ appropriately.

\begin{definition}[\textbf{Amenable hypergraph transformation, quotient hypergraph transformation}]\label{def:AmenQuotHT}
Let $\T := (\X,\pi,\SC)$ be a hypergraph transformation on $\sX$ with $\SC$-maximal subsets $\{\D_X\}_{X \in \X}$, and let $R_{V(\sX)}$ be an equivalence relation on $V(\sX)$. Define:
\begin{enumerate}[label={(\arabic*)}]
\item $\SC/R_{V(\sX)} := \{\, S//R_{V(\sX)} \mid S \in \SC \,\}$. \label{def:AmenQuotHT1}
\item $\D_X/R_{V(\sX)} := \{\, S//R_{V(\sX)} \mid S \in \D_X \,\}$ for $X//R_{V(\sX)} \in \X/R_{V(\sX)}$. \label{def:AmenQuotHT2}
\item $\pi/R_{V(\sX)} \colon \X/R_{V(\sX)} \to \sX/R_{V(\sX)}$ such that $\pi/R_{V(\sX)} \big(X//R_{V(\sX)}\big) := \pi(X)//R_{V(\sX)}$ for all $X//R_{V(\sX)} \in \X/R_{V(\sX)}$. \label{def:AmenQuotHT3}
\end{enumerate}
If $\T/R_{V(\sX)} := (\X/R_{V(\sX)},\pi/R_{V(\sX)},\SC/R_{V(\sX)})$ is a hypergraph transformation on the hypergraph family $\sX/R_{V(\sX)}$ then $\T$ is \emph{amenable} with respect to $R_{V(\sX)}$, and in this case $\T/R_{V(\sX)}$ is the \emph{quotient hypergraph transformation} with respect to $R_{V(\sX)}$.
\end{definition}

\begin{remark} For $\T/R_{V(\sX)}$ to be a hypergraph transformation it is necessary that the following hold: the map $\pi/R_{V(\sX)}$ is well defined, that is $X//R_{V(\sX)} = Y//R_{V(\sX)}$ implies $\pi(X)//R_{V(\sX)} = \pi(Y)//R_{V(\sX)}$ for all $X$, $Y \in \X$; the subsets $\D_X/R_{V(\sX)}$ are well defined, that is $X//R_{V(\sX)} = Y//R_{V(\sX)}$ implies $\D_X/R_{V(\sX)} = \D_Y/R_{V(\sX)}$ for all $X$, $Y \in \X$; and, $\T/R_{V(\sX)}$ satisfies Conditions~\ref{Ht1}--\ref{Ht3} in Definition~\ref{def:Ht}.
\end{remark}

The following proposition demonstrates that commutativity of the partial transformations underlying a sequence of amenable hypergraph transformations induces commutativity of the partial transformations underlying the corresponding sequence of quotient hypergraph transformations.

%
\begin{proposition}\label{Proposition:CommQuotHT}
Suppose $\big(\T_i := (\X_i,\pi_i,\SC^i)\big)_{i=1}^n$ is a sequence of hypergraph transformations on $\sX$, where $n \in \NN$ with $n \ge 2$, and $R_{V(\sX)}$ is an equivalence relation on $V(\sX)$ such that each hypergraph transformation in $(\T_i)_{i=1}^n$ is amenable with respect to $R_{V(\sX)}$. Then $\Coin\big((\bigcirc_{i \in \sigma} \pi_{n+1-i})_{\sigma \in S_n}\big)/R_{V(\sX)} \subseteq \Coin\big((\bigcirc_{i \in \sigma} \pi_{n+1-i}/R_{V(\sX)})_{\sigma \in S_n}\big)$.
\end{proposition}

\begin{proof}
Let $X//R_{V(\sX)} \in \Coin\big((\bigcirc_{i \in \sigma} \pi_{n+1-i})_{\sigma \in S_n}\big)/R_{V(\sX)}$ with $X \in \Coin\big((\bigcirc_{i \in \sigma} \pi_{n+1-i})_{\sigma \in S_n}\big)$, and let $\sigma$, $\sigma' \in S_n$. Then
\begin{align*}
\bigcirc_{i \in \sigma} \pi_{n+1-i}/R_{V(\sX)} (X//R_{V(\sX)}) &= \bigcirc_{i \in \sigma} \pi_{n+1-i} (X) //R_{V(\sX)} = \bigcirc_{i \in \sigma'} \pi_{n+1-i} (X) //R_{V(\sX)}\\ &= \bigcirc_{i \in \sigma'} \pi_{n+1-i}/R_{V(\sX)} (X//R_{V(\sX)}),
\end{align*}
where the first and third equalities follow from the definition of a quotient hypergraph transformation and from $\Dom(\pi_j/R_{V(\sX)}) = \Dom(\pi_j)/R_{V(\sX)}$ for all $j \in [n]$, and the second equality follows from the commutativity of the partial transformations $(\pi_j)_{j=1}^n$ on $\Coin\big((\bigcirc_{i \in \sigma} \pi_{n+1-i})_{\sigma \in S_n}\big)$. Therefore, $X//R_{V(\sX)} \in \Coin\big((\bigcirc_{i \in \sigma} \pi_{n+1-i}/R_{V(\sX)})_{\sigma \in S_n}\big)$.
\end{proof}

In Subsection~\ref{subsec:QuotTransformations_Ex} we discuss two examples of amenable hypergraph transformations. Amenability of a given hypergraph transformation may be realised if the equivalence relation on $V(\sX)$ satisfies certain properties. Two examples are $\SC$-preserving equivalence relations and $W$-disjointness preserving equivalence relations:

\begin{definition}[\textbf{$\SC$-preserving, $W$-disjointness preserving}]\leavevmode \label{def:presER}
Suppose $R_{V(\sX)}$ is an equivalence relation on $V(\sX)$, and $\X \subseteq \sX$.
\begin{enumerate}[label={(\arabic*)}]
\item If $\SC \subseteq \X$ then $R_{V(\sX)}$ is \emph{$\SC$-preserving with respect to $\X$} when $\C(S) \subseteq \C(X)$ implies $\C(S//R_{V(\sX)}) \subseteq \C(X//R_{V(\sX)})$ for all $S \in \SC$ and $X \in \X$. \label{def:SpresER}
\item If $W \in \sX$ then $R_{V(\sX)}$ is \emph{$W$-disjointness preserving with respect to $\X$} when $V(W) \cap V(X) = \emptyset$ implies $V(W//R_{V(\sX)}) \cap V(X//R_{V(\sX)}) = \emptyset$ for all $X \in \X$. \label{def:DpresER}
\end{enumerate}
\end{definition}

\begin{remark}
Suppose $\T := (\X,\pi,\SC)$ is a hypergraph transformation on $\sX$ with $\SC$-maximal subsets $\{\D_X\}_{X \in \X}$. If $\T$ is amenable with respect to the equivalence relation $R_{V(\sX)}$ on $V(\sX)$ then by Property~\ref{def:AmenQuotHT2} in Definition~\ref{def:AmenQuotHT} it follows that $S \in \D_X$, whereby $\C(S) \subseteq \C(X)$, implies $\C(S//R_{V(\sX)}) \subseteq \C(X//R_{V(\sX)})$. Part~\ref{def:SpresER} of Definition~\ref{def:presER} is therefore a more general property than this observation.
\end{remark}

\subsection{Examples of quotient hypergraph transformations}\label{subsec:QuotTransformations_Ex}
Here we give two examples of amenable hypergraph transformations, involving hyperedge addition and hypergraph addition.

\subsubsection{Quotients of hyperedge addition hypergraph transformations}\label{subsubsec:QTransformations_Ex_HE}
Our example of an amenable hypergraph transformation involving hyperedge addition is based on a notion of hyperedge equivalence with respect to an equivalence relation on $V(\sX)$.

\begin{definition}[\textbf{Vertex-augmented quotient of a hyperedge}]\leavevmode \label{def:vertaugquot}
Suppose $R_{V(\sX)}$ is an equivalence relation on $V(\sX)$, and $e \in E(\sX)$. The \emph{vertex-augmented quotient of the hyperedge} $e$ with respect to $R_{V(\sX)}$ is defined by $e//R_{V(\sX)} := \big\{\, [v]_{R_{V(\sX)}} \mid v \in e \,\big\} \in E(\sX/R_{V(\sX)})$, where we assume $\lvert \big\{\, [v]_{R_{V(\sX)}} \mid v \in e \,\big\} \rvert \ge 2$ if we disallow loops.
\end{definition}

\begin{remark}\leavevmode
\begin{enumerate}[label={(\arabic*)}]
\item To see that $e//R_{V(\sX)} \in E(\sX/R_{V(\sX)})$, note that $e \in E(\sX)$ implies $e \in E(X)$ for some $X \in \sX$, so $e//R_{V(\sX)} \in E(X//R_{V(\sX)})$, noting that $\lvert e \rvert \ge 2$ and $\lvert e//R_{V(\sX)} \rvert \ge 2$ if we disallow loops, hence $e//R_{V(\sX)} \in E(\sX/R_{V(\sX)})$.
\item The notation $e//R_{V(\sX)}$ is consistent with the definition of a vertex-augmented quotient hypergraph: if $X$ is a hypergraph with $f \in E(X)$ then, assuming $\lvert f//R_{V(\sX)} \rvert \ge 2$ if we disallow loops, the quotient map sends $f \in E(X)$ to $f//R_{V(\sX)} \in E(X//R_{V(\sX)})$.
\end{enumerate}
\end{remark}

\begin{definition}[\textbf{$e$-equivalent hyperedges}]\label{def:RequivEdges}
Suppose $R_{V(\sX)}$ is an equivalence relation on $V(\sX)$, and $e \in E(\sX)$ with $e//R_{V(\sX)} \in E(\sX/R_{V(\sX)})$. Then $\widetilde{e} := \{\, f \in E(\sX) \mid f//R_{V(\sX)} = e//R_{V(\sX)} \,\}$ is the set of \emph{$e$-equivalent hyperedges} with respect to $R_{V(\sX)}$. Note that for $f \in \widetilde{e}$ we have $\widetilde{f} = \widetilde{e}$.
\end{definition}

\begin{notation} Suppose $R_{V(\sX)}$ is an equivalence relation on $V(\sX)$, $e \in E(\sX)$ with $e//R_{V(\sX)} \in E(\sX/R_{V(\sX)})$, and $X \in \sX$. We denote $\widetilde{e}_X := \{\, f \in \widetilde{e} \mid f \subseteq V(X) \,\}$.
\end{notation}

In the following proposition we characterise sets of $e$-equivalent hyperedges.

\begin{proposition}\label{prop:EquivEdges}
Suppose $R_{V(\sX)}$ is an equivalence relation on $V(\sX)$, and $e \in E(\sX)$ with $e//R_{V(\sX)} \in E(\sX/R_{V(\sX)})$. Then the following hold for all $X$, $Y \in \sX$:
\begin{enumerate}[label={(\arabic*)}]
\item $e//R_{V(\sX)} \subseteq V(X//R_{V(\sX)})$ if and only if $\widetilde{e}_X \ne \emptyset$. \label{prop:EquivEdges1}
\item $e//R_{V(\sX)} \in E(X//R_{V(\sX)})$ if and only if $\widetilde{e}_X \cap E(X) \ne \emptyset$. \label{prop:EquivEdges2}
\item If $\C(X) \subseteq \C(Y)$ then $\widetilde{e}_X \subseteq \widetilde{e}_Y$. \label{prop:EquivEdges3}
\item If $X = Y \wedge \widetilde{e}_Y$ then $\widetilde{e}_X = \widetilde{e}_Y$ and $\widetilde{e}_X \cap E(X) = \widetilde{e}_Y \cap E(Y)$. \label{prop:EquivEdges4}
\item If $V(X//R_{V(\sX)}) = V(Y//R_{V(\sX)})$ then $\widetilde{e}_X \ne \emptyset$ if and only if $\widetilde{e}_Y \ne \emptyset$. \label{prop:EquivEdges5}
\item If $E(X//R_{V(\sX)}) = E(Y//R_{V(\sX)})$ then $\widetilde{e}_X \cap E(X) \ne \emptyset$ if and only if $\widetilde{e}_Y \cap E(Y) \ne \emptyset$. \label{prop:EquivEdges6}
\item If $\SC \subseteq \X \subseteq \sX$, $X \in \X$, $\widetilde{e}_X \ne \emptyset$, $X \wedge \widetilde{e}_X \in \SC$, and $R_{V(\sX)}$ is $\SC$-preserving with respect to $\X$, then $(X \wedge \widetilde{e}_X)//R_{V(\sX)} = X//R_{V(\sX)} \wedge e//R_{V(\sX)}$. \label{prop:EquivEdges7}
\end{enumerate}
\end{proposition}

\begin{proof}
\ref{prop:EquivEdges1} $e//R_{V(\sX)} \subseteq V(X//R_{V(\sX)})$ if and only if there exists $f \subseteq V(X)$ such that $f//R_{V(\sX)} = e//R_{V(\sX)}$, if and only if $f \in \widetilde{e}_X$.

\ref{prop:EquivEdges2} If $e//R_{V(\sX)} \in E(X//R_{V(\sX)})$ then there exists $f \in E(X)$ such that $f//R_{V(\sX)} = e//R_{V(\sX)}$, so $f \in \widetilde{e}_X \cap E(X)$. Conversely, if $f \in \widetilde{e}_X \cap E(X)$ then $f//R_{V(\sX)} \in E(X//R_{V(\sX)})$, since $f//R_{V(\sX)} = e//R_{V(\sX)} \in E(\sX/R_{V(\sX)})$, and hence $e//R_{V(\sX)} = f//R_{V(\sX)} \in E(X//R_{V(\sX)})$.

\ref{prop:EquivEdges3} If $f \in \widetilde{e}_X$ then $f \in \widetilde{e}$ and $f \subseteq V(X)$, so $\C(X) \subseteq \C(Y)$ implies $f \subseteq V(X) \subseteq V(Y)$, therefore $f \in \widetilde{e}_Y$.

\ref{prop:EquivEdges4} Since $\C(X) \subseteq \C(Y)$ it follows from Part~\ref{prop:EquivEdges3} of this proposition that $\widetilde{e}_X \subseteq \widetilde{e}_Y$. For the reverse inclusion, let $f \in \widetilde{e}_Y$ so that $f \in \widetilde{e}$ and $f \subseteq V(Y)$. Noting that $\C(Y \wedge f) \subseteq \C(Y \wedge \widetilde{e}_Y)$, we have $f \subseteq V(Y)$ implies $f \subseteq V(Y \wedge f)$ implies $f \subseteq V(Y \wedge \widetilde{e}_Y) = V(X)$, hence $f \in \widetilde{e}_X$. Therefore $\widetilde{e}_X = \widetilde{e}_Y$.

Now, $\widetilde{e}_Y \cap E(Y) = \widetilde{e}_X \cap E(Y) \subseteq \widetilde{e}_X \cap E(X)$. For the reverse inclusion, if $f \in \widetilde{e}_X \cap E(X)$ then $f \in E(Y)$, so $\widetilde{e}_X \cap E(X) \subseteq \widetilde{e}_Y \cap E(Y)$. Therefore $\widetilde{e}_X \cap E(X) = \widetilde{e}_Y \cap E(Y)$.

\ref{prop:EquivEdges5} Using Part~\ref{prop:EquivEdges1} of this proposition, $\widetilde{e}_X \ne \emptyset$ if and only if $e//R_{V(\sX)} \subseteq V(X//R_{V(\sX)})$ if and only if $e//R_{V(\sX)} \subseteq V(Y//R_{V(\sX)})$ if and only if $\widetilde{e}_Y \ne \emptyset$.

\ref{prop:EquivEdges6} Using Part~\ref{prop:EquivEdges2} of this proposition, $\widetilde{e}_X \cap E(X) \ne \emptyset$ if and only if $e//R_{V(\sX)} \in E(X//R_{V(\sX)})$ if and only if $e//R_{V(\sX)} \in E(Y//R_{V(\sX)})$ if and only if $\widetilde{e}_Y \cap E(Y) \ne \emptyset$.

\ref{prop:EquivEdges7} We show $\C\big((X \wedge \widetilde{e}_X)//R_{V(\sX)}\big) = \C(X//R_{V(\sX)} \wedge e//R_{V(\sX)})$. Since $\C(X \wedge \widetilde{e}_X) \subseteq \C(X)$ and $R_{V(\sX)}$ is $\SC$-preserving with respect to $\X$ we have $\C\big((X \wedge \widetilde{e}_X)//R_{V(\sX)}\big) \subseteq \C(X//R_{V(\sX)})$. If $C' \in \C\big((X \wedge \widetilde{e}_X)//R_{V(\sX)}\big)$ then $C' = \bigcup_{C \in A} C//R_{V(\sX)}$ for some nonempty subset $A \subseteq \C(X \wedge \widetilde{e}_X)$, and since for all $C \in \C(X \wedge \widetilde{e}_X)$ there exists $f \in \widetilde{e}_X$ such that $f \cap V(C) \ne \emptyset$ it follows that $e//R_{V(\sX)} \cap V(C') \ne \emptyset$, so $C' \in \C(X//R_{V(\sX)} \wedge e//R_{V(\sX)})$. Hence $\C\big((X \wedge \widetilde{e}_X)//R_{V(\sX)}\big) \subseteq \C(X//R_{V(\sX)} \wedge e//R_{V(\sX)})$. For the reverse inclusion we prove the contrapositive, so let $C'' \in \C(X//R_{V(\sX)}) \setminus \C\big((X \wedge \widetilde{e}_X)//R_{V(\sX)}\big)$ and we show that $C'' \in \C(X//R_{V(\sX)}) \setminus \C(X//R_{V(\sX)} \wedge e//R_{V(\sX)})$. Now, $C'' \in \C(X//R_{V(\sX)})$ implies $C'' = \bigcup_{C \in B} C//R_{V(\sX)}$ for some nonempty subset $B \subseteq \C(X)$. Since $C''$ and $\C\big((X \wedge \widetilde{e}_X)//R_{V(\sX)}\big)$ are vertex disjoint, we must have $f \cap V(C) = \emptyset$ for all $C \in B$ and for all $f \in \widetilde{e}_X$. Note that if $e//R_{V(\sX)} \cap V(C'') \ne \emptyset$ then $f//R_{V(\sX)} \cap V(C'') \ne \emptyset$ for all $f \in \widetilde{e}_X$, so there must exist $g \in \widetilde{e}_X$ and $C \in B$ such that $g \cap V(C) \ne \emptyset$. It follows that $e//R_{V(\sX)} \cap V(C'') = \emptyset$, hence $C'' \notin \C(X//R_{V(\sX)} \wedge e//R_{V(\sX)})$.
\end{proof}

\begin{definition}[\textbf{Equivalent-hyperedges addition partial transformation}]
Suppose $R_{V(\sX)}$ is an equivalence relation on $V(\sX)$, $e \in E(\sX)$ with $e//R_{V(\sX)} \in E(\sX/R_{V(\sX)})$, and $\X \subseteq \sX$. We define the \emph{equivalent-hyperedges addition partial transformation} $\pi_{\widetilde{e}} \colon \X \to \sX$ such that, for $X \in \X$,
\begin{equation}
\pi_{\widetilde{e}} (X) =
\begin{cases}
X \boxplus \widetilde{e}_X & \text{if $\widetilde{e}_X \ne \emptyset$ and $\widetilde{e}_X \cap E(X) = \emptyset$},\\
X & \text{if $\widetilde{e}_X = \emptyset$ or $\widetilde{e}_X \cap E(X) \ne \emptyset$}.
\end{cases}
\end{equation}
Therefore, if $\widetilde{e}_X \ne \emptyset$ and $\widetilde{e}_X \cap E(X) = \emptyset$ then $\pi_{\widetilde{e}}$ adds all of the hyperedges in $\widetilde{e}_X$ to $X$.
\end{definition}

\begin{definition}[\textbf{$\widetilde{e}$-closed for addition, $\widetilde{e}$-amenable for addition}]
Suppose $R_{V(\sX)}$ is an equivalence relation on $V(\sX)$, $e \in E(\sX)$ with $e//R_{V(\sX)} \in E(\sX/R_{V(\sX)})$, and $\X \subseteq \sX$. We say that $\X$ is \emph{$\widetilde{e}$-closed for addition} if $X \in \X$, $\widetilde{e}_X \ne \emptyset$, and $\widetilde{e}_X \cap E(X) = \emptyset$ imply $X \wedge \widetilde{e}_X \in \X$. Further, we say that $\X$ is \emph{$\widetilde{e}$-amenable for addition} if $S \in \{\, S \in \X \mid \text{$\widetilde{e}_S \ne \emptyset$, $\widetilde{e}_S \cap E(S) = \emptyset$, and $S = S \wedge \widetilde{e}_S$ } \,\}$ implies $\C(S) \nsubseteq \C(X)$ for all $X \in \X$ with $\widetilde{e}_X \cap E(X) \ne \emptyset$.
\end{definition}

\begin{proposition}\label{prop:eHEat}
Suppose $R_{V(\sX)}$ is an equivalence relation on $V(\sX)$, $e \in E(\sX)$ with $e//R_{V(\sX)} \in E(\sX/R_{V(\sX)})$, and $\X \subseteq \sX$ is both $\widetilde{e}$-closed for addition and $\widetilde{e}$-amenable for addition. Let $\pi_{\widetilde{e}} \colon \X \to \sX$ be the equivalent-hyperedges addition partial transformation, define
\begin{equation} \label{prop:eHEat3}
\widetilde{\SC} := \{\, S \in \X \mid \text{$\widetilde{e}_S \ne \emptyset$, $\widetilde{e}_S \cap E(S) = \emptyset$, and $S = S \wedge \widetilde{e}_S$ } \,\},
\end{equation}
and for $X \in \X$ define
\begin{equation} \label{prop:eHEat4}
\widetilde{\D}_X :=
\begin{cases}
\{X \wedge \widetilde{e}_X\}, &\enspace\text{when $\widetilde{e}_X \ne \emptyset$ and $\widetilde{e}_X \cap E(X) = \emptyset$},\\
\emptyset, &\enspace\text{when $\widetilde{e}_X = \emptyset$ or $\widetilde{e}_X \cap E(X) \ne \emptyset$}.
\end{cases}
\end{equation}
Then:
\begin{enumerate}[label={(\arabic*)}]
\item $\T_{\widetilde{e}} := (\X,\pi_{\widetilde{e}},\widetilde{\SC})$ is a hypergraph transformation on $\sX$ with $\widetilde{\SC}$-maximal subsets $\{\widetilde{\D}_X\}_{X \in \X}$. \label{prop:eHEat6}
\item If $R_{V(\sX)}$ is $\widetilde{\SC}$-preserving with respect to $\X$ then $\T^+_{e//R_{V(\sX)}} := (\X/R_{V(\sX)},\pi^+_{e//R_{V(\sX)}},\SC^+)$ is a hyperedge addition hypergraph transformation on $\sX/R_{V(\sX)}$ for the hyperedge $e//R_{V(\sX)}$, where $\SC^+ = \{\, T \in \X/R_{V(\sX)} \mid \text{$e//R_{V(\sX)} \subseteq V(T)$, $e//R_{V(\sX)} \notin E(T)$, and $T = T \wedge e//R_{V(\sX)}$} \,\}$. For $X//R_{V(\sX)} \in \X/R_{V(\sX)}$ the $\SC^+$-maximal subsets are $\D^+_{X//R_{V(\sX)}} = \{X//R_{V(\sX)} \wedge e//R_{V(\sX)}\}$ if $e//R_{V(\sX)} \subseteq V(X//R_{V(\sX)})$ and $e//R_{V(\sX)} \notin E(X//R_{V(\sX)})$, and $\D^+_{X//R_{V(\sX)}} = \emptyset$ if $e//R_{V(\sX)} \nsubseteq V(X//R_{V(\sX)})$ or $e//R_{V(\sX)} \in E(X//R_{V(\sX)})$. \label{prop:eHEat7}
\item $\widetilde{\SC}/R_{V(\sX)}$ is upward closed with respect to $\SC^+$. \label{prop:eHEat8}
\item If $R_{V(\sX)}$ is $\widetilde{\SC}$-preserving with respect to $\X$ then $\T_{\widetilde{e}}$ is amenable with respect to $R_{V(\sX)}$. In particular, the hypergraph transformation $\T_{\widetilde{e}}/R_{V(\sX)} := (\X/R_{V(\sX)},\pi_{\widetilde{e}}/R_{V(\sX)},\widetilde{\SC}/R_{V(\sX)})$ is the support reduction of $\T^+_{e//R_{V(\sX)}}$ corresponding to $\widetilde{\SC}/R_{V(\sX)}$. \label{prop:eHEat9}
\end{enumerate}
\end{proposition}

\begin{proof}
\ref{prop:eHEat6} For nonredundancy, if $S \in \widetilde{\SC}$ then $\pi_{\widetilde{e}}(S) = S \boxplus \widetilde{e}_S$ modifies all components of $S$ by addition of the hyperedges in $\widetilde{e}_S$, since $S = S \wedge \widetilde{e}_S$, hence $\C(S) \cap \C\big(\pi_{\widetilde{e}}(S)\big) = \emptyset$; additionally, $\N \notin \widetilde{\SC}$.

To see that $\widetilde{\SC}$ is component maximal, let $X \in \X$. Conditions~\ref{def:CompMax1} to \ref{def:CompMax3} of Definition~\ref{def:CompMax} follow immediately from the definition of $\widetilde{\D}_X$, and for Condition~\ref{def:CompMax4} suppose $S \in \widetilde{\SC}$ and $\C(S) \subseteq \C(X)$. Part~\ref{prop:EquivEdges3} of Proposition~\ref{prop:EquivEdges} implies $\widetilde{e}_S \subseteq \widetilde{e}_X$, and hence $\widetilde{e}_X \ne \emptyset$ since $\widetilde{e}_S \ne \emptyset$. Moreover, since $\X$ is $\widetilde{e}$-amenable for addition it follows that $\widetilde{e}_X \cap E(X) = \emptyset$. So $\widetilde{\D}_X = \{X \wedge \widetilde{e}_X\}$. Now, $\C(S \wedge \widetilde{e}_S) \subseteq \C(X \wedge \widetilde{e}_X)$, since $\widetilde{e}_S \subseteq \widetilde{e}_X$ and $\C(S) \subseteq \C(X)$, therefore $\C(S) \subseteq \C(X \wedge \widetilde{e}_X)$, since $S = S \wedge \widetilde{e}_S$.

To show the direct sum decomposition is preserved, let $X \in \X$. Note that $\widetilde{\SC}_X = \widetilde{\D}_X$, so $\lvert \widetilde{\SC}_X \rvert = \lvert \pi_{\widetilde{e}}(\widetilde{\SC}_X) \rvert \le 1$, and it also holds trivially that $\pi_{\widetilde{e}}(\widetilde{\SC}_X)$ consists of pairwise vertex-disjoint hypergraphs. If $\widetilde{e}_X \ne \emptyset$ and $\widetilde{e}_X \cap E(X) = \emptyset$ then $\widetilde{\SC}_X = \widetilde{\D}_X = \{S\}$ where $S = X \wedge \widetilde{e}_X$, and $\widetilde{e}_S = \widetilde{e}_X$ by Part~\ref{prop:EquivEdges4} of Proposition~\ref{prop:EquivEdges}, so $\pi_{\widetilde{e}} (X) = X \boxplus \widetilde{e}_X = (X \ominus S) \oplus (S \boxplus \widetilde{e}_X) = (X \ominus S) \oplus (S \boxplus \widetilde{e}_S) = \widebar{X} \oplus \pi_{\widetilde{e}} (S)$, where $\widebar{X} = X \ominus S$. If $\widetilde{e}_X = \emptyset$ or $\widetilde{e}_X \cap E(X) \ne \emptyset$ then $\widetilde{\SC}_X = \widetilde{\D}_X = \emptyset$, so $\pi_{\widetilde{e}} (X) = X = \widebar{X}$. In any case we have $\pi_{\widetilde{e}} (X) = \widebar{X} \oplus \big(\bigoplus_{S \in \widetilde{\SC}_X} \pi_{\widetilde{e}}(S)\big)$, where $\widebar{X} := X \ominus \big(\bigoplus_{S \in \widetilde{\SC}_X} S\big)$. We conclude that $(\X,\pi_{\widetilde{e}},\widetilde{\SC})$ is a hypergraph transformation on $\sX$ with $\widetilde{\SC}$-maximal subsets $\{\widetilde{\D}_X\}_{X \in \X}$.

\ref{prop:eHEat7} Note that $e//R_{V(\sX)} \in E(\sX/R_{V(\sX)})$ and $\X/R_{V(\sX)} \subseteq \sX/R_{V(\sX)}$, so $\pi^+_{e//R_{V(\sX)}} \colon \X/R_{V(\sX)} \to \sX/R_{V(\sX)}$ is a hyperedge addition partial transformation. We show that $\X/R_{V(\sX)}$ is $e//R_{V(\sX)}$-closed for addition, and then the result follows from Part~\ref{prop:HEtadd5} of Proposition~\ref{prop:HEtadd}. Suppose $X//R_{V(\sX)} \in \X/R_{V(\sX)}$ with $X \in \X$, $e//R_{V(\sX)} \subseteq V(X//R_{V(\sX)})$, and $e//R_{V(\sX)} \notin E(X//R_{V(\sX)})$. Then Parts~\ref{prop:EquivEdges1} and \ref{prop:EquivEdges2} of Proposition~\ref{prop:EquivEdges} imply $\widetilde{e}_X \ne \emptyset$ and $\widetilde{e}_X \cap E(X) = \emptyset$, respectively, so since $\X$ is $\widetilde{e}$-closed for addition we have $S := X \wedge \widetilde{e}_X \in \X$. Since $\widetilde{e}_S = \widetilde{e}_X$ by Part~\ref{prop:EquivEdges4} of Proposition~\ref{prop:EquivEdges}, since $\widetilde{e}_X \cap E(X) = \emptyset$ implies $\widetilde{e}_S \cap E(S) = \emptyset$, and since $S \wedge \widetilde{e}_S = S \wedge \widetilde{e}_X = S$, it follows that $S \in \widetilde{\SC}$. So $X//R_{V(\sX)} \wedge e//R_{V(\sX)} = S//R_{V(\sX)} \in \X/R_{V(\sX)}$, where the equality holds by Part~\ref{prop:EquivEdges7} of Proposition~\ref{prop:EquivEdges} since $R_{V(\sX)}$ is $\widetilde{\SC}$-preserving with respect to $\X$.

\ref{prop:eHEat8} We show that $\widetilde{\SC}/R_{V(\sX)} \subseteq \SC^+$, so let $S//R_{V(\sX)} \in \widetilde{\SC}/R_{V(\sX)}$ with $S \in \widetilde{\SC}$, noting $S//R_{V(\sX)} \in \X/R_{V(\sX)}$. Then $\widetilde{e}_S \ne \emptyset$ implies $e//R_{V(\sX)} \subseteq V(S//R_{V(\sX)})$, $\widetilde{e}_S \cap E(S) = \emptyset$ implies $e//R_{V(\sX)} \notin E(S//R_{V(\sX)})$, and $S = S \wedge \widetilde{e}_S$ implies $S//R_{V(\sX)} = (S \wedge \widetilde{e}_S)//R_{V(\sX)} = S//R_{V(\sX)} \wedge e//R_{V(\sX)}$ by Parts~\ref{prop:EquivEdges1}, \ref{prop:EquivEdges2}, and \ref{prop:EquivEdges7} of Proposition~\ref{prop:EquivEdges}, respectively. Hence $S//R_{V(\sX)} \in \SC^+$, and we conclude that $\widetilde{\SC}/R_{V(\sX)} \subseteq \SC^+$.

To see that $\widetilde{\SC}/R_{V(\sX)}$ is upward closed with respect to $\SC^+$, suppose $S//R_{V(\sX)} \in \widetilde{\SC}/R_{V(\sX)}$ with $S \in \widetilde{\SC}$, $T//R_{V(\sX)} \in \SC^+$, and $\C(S//R_{V(\sX)}) \subseteq \C(T//R_{V(\sX)})$. We show that $T//R_{V(\sX)} \in \widetilde{\SC}/R_{V(\sX)}$. Note that $S \in \widetilde{\SC}$ implies $\widetilde{e}_S \ne \emptyset$ implies $e//R_{V(\sX)} \subseteq V(S//R_{V(\sX)})$, by Part~\ref{prop:EquivEdges1} of Proposition~\ref{prop:EquivEdges}, and also $e//R_{V(\sX)} \subseteq V(T//R_{V(\sX)})$. So $C \in \C(T//R_{V(\sX)}) \setminus \C(S//R_{V(\sX)})$ implies $e//R_{V(\sX)} \cap V(C) = \emptyset$, hence $T//R_{V(\sX)} \ne T//R_{V(\sX)} \wedge e//R_{V(\sX)}$, contradicting $T//R_{V(\sX)} \in \SC^+$. It follows that $S//R_{V(\sX)} = T//R_{V(\sX)}$, hence $T//R_{V(\sX)} \in \widetilde{\SC}/R_{V(\sX)}$. Note that it is not necessarily true that $T \in \widetilde{\SC}$. We conclude that $\widetilde{\SC}/R_{V(\sX)}$ is upward closed with respect to $\SC^+$.

\ref{prop:eHEat9} We show that $\T_{\widetilde{e}}/R_{V(\sX)}$ is a well defined hypergraph transformation, first showing that the partial transformation $\pi_{\widetilde{e}}/R_{V(\sX)}$ is well defined, and second showing that $\T_{\widetilde{e}}/R_{V(\sX)}$ satisfies Conditions~\ref{Ht1}--\ref{Ht3} in Definition~\ref{def:Ht}.

First, let $X//R_{V(\sX)}$, $Y//R_{V(\sX)} \in \X/R_{V(\sX)}$ with $X//R_{V(\sX)} = Y//R_{V(\sX)}$. Then $\widetilde{e}_X \ne \emptyset$ if and only if $\widetilde{e}_Y \ne \emptyset$ by Part~\ref{prop:EquivEdges5} of Proposition~\ref{prop:EquivEdges}, and $\widetilde{e}_X \cap E(X) \ne \emptyset$ if and only if $\widetilde{e}_Y \cap E(Y) \ne \emptyset$ by Part~\ref{prop:EquivEdges6} of Proposition~\ref{prop:EquivEdges}. So $\widetilde{e}_X = \emptyset$ or $\widetilde{e}_X \cap E(X) \ne \emptyset$ implies $\pi_{\widetilde{e}}/R_{V(\sX)} (X//R_{V(\sX)}) = \pi_{\widetilde{e}}(X)//R_{V(\sX)} = X//R_{V(\sX)} = Y//R_{V(\sX)} = \pi_{\widetilde{e}}(Y)//R_{V(\sX)} = \pi_{\widetilde{e}}/R_{V(\sX)} (Y//R_{V(\sX)})$. Further, $\widetilde{e}_X \ne \emptyset$ and $\widetilde{e}_X \cap E(X) = \emptyset$ implies $\pi_{\widetilde{e}}/R_{V(\sX)} (X//R_{V(\sX)}) = \pi_{\widetilde{e}}(X)//R_{V(\sX)} = (X \boxplus \widetilde{e}_X)//R_{V(\sX)} = X//R_{V(\sX)} \boxplus e//R_{V(\sX)} = Y//R_{V(\sX)} \boxplus e//R_{V(\sX)} = (Y \boxplus \widetilde{e}_Y)//R_{V(\sX)} = \pi_{\widetilde{e}}(Y)//R_{V(\sX)} = \pi_{\widetilde{e}}/R_{V(\sX)} (Y//R_{V(\sX)})$. It follows that $\pi_{\widetilde{e}}/R_{V(\sX)}$ is well defined.

Second, nonredundancy holds since if $S//R_{V(\sX)} \in \widetilde{\SC}/R_{V(\sX)}$ with $S \in \widetilde{\SC}$ then $\pi_{\widetilde{e}}/R_{V(\sX)}(S//R_{V(\sX)}) = \pi_{\widetilde{e}}(S)//R_{V(\sX)} = (S \boxplus \widetilde{e}_S)//R_{V(\sX)} = S//R_{V(\sX)} \boxplus e//R_{V(\sX)}$, so $\C(S//R_{V(\sX)}) \cap \C\big(\pi_{\widetilde{e}}/R_{V(\sX)}(S//R_{V(\sX)})\big) = \emptyset$ since addition of the hyperedge $e//R_{V(\sX)}$ modifies all components of $S//R_{V(\sX)}$; additionally, $\N \notin \widetilde{\SC}/R_{V(\sX)}$.

Further, $\widetilde{\SC}/R_{V(\sX)}$ is upward closed with respect to $\SC^+$ by Part~\ref{prop:eHEat8} of this proposition, so Part~\ref{prop:UpCl3} of Proposition~\ref{prop:UpCl} implies that $\widetilde{\SC}/R_{V(\sX)}$ is component maximal with $\widetilde{\SC}/R_{V(\sX)}$-maximal subsets defined by $\D^+_{X//R_{V(\sX)}} \cap \widetilde{\SC}/R_{V(\sX)}$ for $X//R_{V(\sX)} \in \X/R_{V(\sX)}$. Let $X//R_{V(\sX)} \in \X/R_{V(\sX)}$. We show that the subset $\widetilde{\D}_X/R_{V(\sX)}$ is $\widetilde{\SC}/R_{V(\sX)}$-maximal, and then uniqueness by Part~\ref{prop:CompMax1} of Proposition~\ref{prop:CompMax} will give $\widetilde{\D}_X/R_{V(\sX)} = \D^+_{X//R_{V(\sX)}} \cap \widetilde{\SC}/R_{V(\sX)}$. Note that each $\widetilde{\D}_X/R_{V(\sX)}$ is well defined: if $X//R_{V(\sX)} = Y//R_{V(\sX)} \in \X/R_{V(\sX)}$ then $\widetilde{e}_X \ne \emptyset$ if and only if $\widetilde{e}_Y \ne \emptyset$, and $\widetilde{e}_X \cap E(X) \ne \emptyset$ if and only if $\widetilde{e}_Y \cap E(Y) \ne \emptyset$, by Parts~\ref{prop:EquivEdges5} and \ref{prop:EquivEdges6} of Proposition~\ref{prop:EquivEdges}, respectively; so either $\widetilde{\D}_X/R_{V(\sX)} = \emptyset = \widetilde{\D}_Y/R_{V(\sX)}$ or, since $R_{V(\sX)}$ is $\widetilde{\SC}$-preserving with respect to $\X$, Part~\ref{prop:EquivEdges7} of Proposition~\ref{prop:EquivEdges} implies $(X \wedge \widetilde{e}_X)//R_{V(\sX)} = X//R_{V(\sX)} \wedge e//R_{V(\sX)} = Y//R_{V(\sX)} \wedge e//R_{V(\sX)} = (Y \wedge \widetilde{e}_Y)//R_{V(\sX)}$ and therefore $\widetilde{\D}_X/R_{V(\sX)} = \{(X \wedge \widetilde{e}_X)//R_{V(\sX)}\} = \widetilde{\D}_Y/R_{V(\sX)}$. Now, Conditions~\ref{def:CompMax1} to \ref{def:CompMax3} of Definition~\ref{def:CompMax} follow from the definition of $\widetilde{\D}_X/R_{V(\sX)}$, noting that $\C\big((X \wedge \widetilde{e}_X)//R_{V(\sX)}\big) = \C(X//R_{V(\sX)} \wedge e//R_{V(\sX)}) \subseteq \C(X//R_{V(\sX)})$ since $R_{V(\sX)}$ is $\widetilde{\SC}$-preserving with respect to $\X$ and applying Part~\ref{prop:EquivEdges7} of Proposition~\ref{prop:EquivEdges}. For Condition~\ref{def:CompMax4} of Definition~\ref{def:CompMax}, suppose $S//R_{V(\sX)} \in \widetilde{\SC}/R_{V(\sX)}$ with $S \in \widetilde{\SC}$, and $\C(S//R_{V(\sX)}) \subseteq \C(X//R_{V(\sX)})$. Since $S \in \widetilde{\SC}$ we have $\widetilde{e}_S \ne \emptyset$ implies $e//R_{V(\sX)} \subseteq V(S//R_{V(\sX)})$, by Part~\ref{prop:EquivEdges1} of Proposition~\ref{prop:EquivEdges}, and $\widetilde{e}_S \cap E(S) = \emptyset$ implies $e//R_{V(\sX)} \notin E(S//R_{V(\sX)})$, by Part~\ref{prop:EquivEdges2} of Proposition~\ref{prop:EquivEdges}. Then $e//R_{V(\sX)} \subseteq V(X//R_{V(\sX)})$ implies $\widetilde{e}_X \ne \emptyset$, by Part~\ref{prop:EquivEdges1} of Proposition~\ref{prop:EquivEdges}, and since $e//R_{V(\sX)} \notin E(S//R_{V(\sX)})$ implies $e//R_{V(\sX)} \notin E(X//R_{V(\sX)})$ we have $\widetilde{e}_X \cap E(X) = \emptyset$ by Part~\ref{prop:EquivEdges2} of Proposition~\ref{prop:EquivEdges}. So $\widetilde{\D}_X/R_{V(\sX)} = \{(X \wedge \widetilde{e}_X)//R_{V(\sX)}\}$. Now, $C \in \C(X//R_{V(\sX)}) \setminus \C(S//R_{V(\sX)})$ implies $e//R_{V(\sX)} \cap V(C) = \emptyset$, hence $S//R_{V(\sX)} = X//R_{V(\sX)} \wedge e//R_{V(\sX)} = (X \wedge \widetilde{e}_X)//R_{V(\sX)} \in \widetilde{\D}_X/R_{V(\sX)}$, where the last equality follows from Part~\ref{prop:EquivEdges7} of Proposition~\ref{prop:EquivEdges} noting that $R_{V(\sX)}$ is $\widetilde{\SC}$-preserving with respect to $\X$. We conclude that $\widetilde{\D}_X/R_{V(\sX)}$ is $\widetilde{\SC}/R_{V(\sX)}$-maximal.

We show the direct sum decomposition is preserved, so let $X//R_{V(\sX)} \in \X/R_{V(\sX)}$. For notational simplicity denote $\SC^{\ast} := \widetilde{\SC}/R_{V(\sX)}$ and $Y := X \wedge \widetilde{e}_X$. Note that Part~\ref{prop:EquivEdges4} of Proposition~\ref{prop:EquivEdges} implies $\widetilde{e}_Y = \widetilde{e}_X$ and $\widetilde{e}_Y \cap E(Y) = \widetilde{e}_X \cap E(X)$. First suppose that $\widetilde{e}_X = \emptyset$ or $\widetilde{e}_X \cap E(X) \ne \emptyset$. Then $\widetilde{\D}_X/R_{V(\sX)} = \emptyset$, so $\SC^{\ast}_{X//R_{V(\sX)}} = \emptyset$, hence $\lvert \SC^{\ast}_{X//R_{V(\sX)}} \rvert = \lvert \pi_{\widetilde{e}}/R_{V(\sX)}(\SC^{\ast}_{X//R_{V(\sX)}}) \rvert = 0$, it holds trivially that $\pi_{\widetilde{e}}/R_{V(\sX)}(\SC^{\ast}_{X//R_{V(\sX)}})$ consists of pairwise vertex-disjoint hypergraphs, and $\pi_{\widetilde{e}}/R_{V(\sX)} (X//R_{V(\sX)}) = X//R_{V(\sX)} = \overline{X//R_{V(\sX)}}$. Second suppose that $\widetilde{e}_X \ne \emptyset$ and $\widetilde{e}_X \cap E(X) = \emptyset$. Then $\widetilde{\D}_X/R_{V(\sX)} = \{Y//R_{V(\sX)}\}$. Now, $\pi_{\widetilde{e}}/R_{V(\sX)}(Y//R_{V(\sX)}) = \pi_{\widetilde{e}}(Y)//R_{V(\sX)} = (Y \boxplus \widetilde{e}_Y)//R_{V(\sX)} = Y//R_{V(\sX)} \boxplus e//R_{V(\sX)}$, where the second equality follows since $\widetilde{e}_Y \ne \emptyset$ and $\widetilde{e}_Y \cap E(Y) = \emptyset$. Then $V\big(\pi_{\widetilde{e}}/R_{V(\sX)}(Y//R_{V(\sX)})\big) \cap V(X//R_{V(\sX)} \ominus Y//R_{V(\sX)}) = \emptyset$, so $\SC^{\ast}_{X//R_{V(\sX)}} = \{Y//R_{V(\sX)}\}$, hence $\lvert \SC^{\ast}_{X//R_{V(\sX)}} \rvert = \lvert \pi_{\widetilde{e}}/R_{V(\sX)}(\SC^{\ast}_{X//R_{V(\sX)}}) \rvert = 1$, and it also holds trivially that $\pi_{\widetilde{e}}/R_{V(\sX)}(\SC^{\ast}_{X//R_{V(\sX)}})$ consists of pairwise vertex-disjoint hypergraphs. Further, $\pi_{\widetilde{e}}/R_{V(\sX)}(X//R_{V(\sX)}) = \pi_{\widetilde{e}}(X)//R_{V(\sX)} = (X \boxplus \widetilde{e}_X)//R_{V(\sX)} = X//R_{V(\sX)} \boxplus e//R_{V(\sX)} = (X//R_{V(\sX)} \ominus Y//R_{V(\sX)}) \oplus (Y//R_{V(\sX)} \boxplus e//R_{V(\sX)}) = \overline{X//R_{V(\sX)}} \oplus \pi_{\widetilde{e}}/R_{V(\sX)} (Y//R_{V(\sX)})$, where $\overline{X//R_{V(\sX)}} := (X//R_{V(\sX)} \ominus Y//R_{V(\sX)})$. In any case we have the decomposition $\pi_{\widetilde{e}}/R_{V(\sX)} (X//R_{V(\sX)}) = \overline{X//R_{V(\sX)}} \oplus \big(\bigoplus_{S//R_{V(\sX)} \in \SC^{\ast}_{X//R_{V(\sX)}}} \pi_{\widetilde{e}}/R_{V(\sX)} (S//R_{V(\sX)})\big)$, where $\overline{X//R_{V(\sX)}} := X//R_{V(\sX)} \ominus \big(\bigoplus_{S//R_{V(\sX)} \in \SC^{\ast}_{X//R_{V(\sX)}}} S//R_{V(\sX)} \big)$.

We conclude that $\T_{\widetilde{e}}/R_{V(\sX)}$ is a hypergraph transformation on $\sX/R_{V(\sX)}$ with $\widetilde{\SC}/R_{V(\sX)}$-maximal subsets $\widetilde{\D}_X/R_{V(\sX)}$ for $X//R_{V(\sX)} \in \X/R_{V(\sX)}$.

To show that $\T_{\widetilde{e}}/R_{V(\sX)}$ is the support reduction of $\T^+_{e//R_{V(\sX)}}$ corresponding to $\widetilde{\SC}/R_{V(\sX)}$ it remains to show that $\pi_{\widetilde{e}}/R_{V(\sX)}$ and $\pi^+_{e//R_{V(\sX)}}$ are equal on $\widetilde{\SC}/R_{V(\sX)}$, so let $S//R_{V(\sX)} \in \widetilde{\SC}/R_{V(\sX)}$ with $S \in \widetilde{\SC}$. Then $\widetilde{e}_S \ne \emptyset$ implies $e//R_{V(\sX)} \subseteq V(S//R_{V(\sX)})$, by Part~\ref{prop:EquivEdges1} of Proposition~\ref{prop:EquivEdges}, and $\widetilde{e}_S \cap E(S) = \emptyset$ implies $e//R_{V(\sX)} \notin E(S//R_{V(\sX)})$, by Part~\ref{prop:EquivEdges2} of Proposition~\ref{prop:EquivEdges}. So $\pi_{\widetilde{e}}/R_{V(\sX)}(S//R_{V(\sX)}) = \pi_{\widetilde{e}}(S)//R_{V(\sX)} = (S \boxplus \widetilde{e}_S)//R_{V(\sX)} = S//R_{V(\sX)} \boxplus e//R_{V(\sX)} = \pi^+_{e//R_{V(\sX)}}(S//R_{V(\sX)})$, as required.
\end{proof}

\subsubsection{Quotients of hypergraph addition hypergraph transformations}\label{subsubsec:QTransformations_Ex_HG}
We now consider an example of an amenable hypergraph transformation involving hypergraph addition.

\begin{lemma} \label{lemma:RequivH}
Suppose $\X \subseteq \sX$, $W \in \sX^{\ast}$, and $R_{V(\sX)}$ is an equivalence relation on $V(\sX)$ that is $W$-disjointness preserving with respect to $\X$. Then the following hold for all $X$, $Y \in \X$:
\begin{enumerate}[label={(\arabic*)}]
\item $V(W) \cap V(X) = \emptyset$ if and only if $V(W//R_{V(\sX)}) \cap V(X//R_{V(\sX)}) = \emptyset$. \label{lemma:RequivH1}
\item If $X//R_{V(\sX)} = Y//R_{V(\sX)}$ then $V(W) \cap V(X) = \emptyset$ if and only if $V(W) \cap V(Y) = \emptyset$. \label{lemma:RequivH2}
\end{enumerate}
\end{lemma}

\begin{proof}
\ref{lemma:RequivH1} Since $R_{V(\sX)}$ is $W$-disjointness preserving with respect to $\X$ the forward direction holds. For the reverse direction, a contrapositive argument gives $V(W) \cap V(X) \ne \emptyset$ implies $V(W//R_{V(\sX)}) \cap V(X//R_{V(\sX)}) \ne \emptyset$.

\ref{lemma:RequivH2} $V(W) \cap V(X) = \emptyset$ if and only if $V(W//R_{V(\sX)}) \cap V(X//R_{V(\sX)}) = \emptyset$ if and only if $V(W//R_{V(\sX)}) \cap V(Y//R_{V(\sX)}) = \emptyset$ if and only if $V(W) \cap V(Y) = \emptyset$, where the first and third equivalences hold by Part\ref{lemma:RequivH1} of this proposition.

\end{proof}

\begin{proposition}\label{prop:HaTq}
Suppose $W \in \sX^{\ast}$, $\X \subseteq \sX$ with $\N \in \X$, and $R_{V(\sX)}$ is an equivalence relation on $V(\sX)$ that is $W$-disjointness preserving with respect to $\X$. Then the hypergraph transformation $\T^+_W = (\X,\pi^+_W,\SC^+)$ is amenable with respect to $R_{V(\sX)}$. In particular, the hypergraph transformation $\T^+_W/R_{V(\sX)} := (\X/R_{V(\sX)},\pi^+_W/R_{V(\sX)},\SC^+/R_{V(\sX)})$ is equal to the hypergraph addition hypergraph transformation $\T^+_{W//R_{V(\sX)}} := (\X/R_{V(\sX)},\pi^+_{W//R_{V(\sX)}},\SC^+)$ where $W//R_{V(\sX)} \in \sX^{\ast}/R_{V(\sX)}$ and $\SC^+ = \{\N\}$.
\end{proposition}

\begin{proof}
Noting that $\SC^+/R_{V(\sX)} = \{\N\} = \SC^+$, it suffices to show that $\pi^+_W/R_{V(\sX)}$ is a well defined partial transformation on $\sX/R_{V(\sX)}$ such that $\pi^+_W/R_{V(\sX)} = \pi^+_{W//R_{V(\sX)}}$.

To show $\pi^+_W/R_{V(\sX)}$ is well defined, let $X//R_{V(\sX)}$, $Y//R_{V(\sX)} \in \X/R_{V(\sX)}$ with $X//R_{V(\sX)} = Y//R_{V(\sX)}$. First, if $V(W) \cap V(X) \ne \emptyset$, equivalently $V(W) \cap V(Y) \ne \emptyset$ by Part~\ref{lemma:RequivH2} of Proposition~\ref{lemma:RequivH}, then $\pi^+_W/R_{V(\sX)} (X//R_{V(\sX)}) = \pi^+_W (X)//R_{V(\sX)} = X//R_{V(\sX)} = Y//R_{V(\sX)} = \pi^+_W (Y)//R_{V(\sX)} = \pi^+_W/R_{V(\sX)} (Y//R_{V(\sX)})$. Second, if $V(W) \cap V(X) = \emptyset$, equivalently $V(W) \cap V(Y) = \emptyset$ by Part\ref{lemma:RequivH2} of Proposition~\ref{lemma:RequivH}, then Part~\ref{lemma:RequivH1} of Proposition~\ref{lemma:RequivH} implies $V(W//R_{V(\sX)}) \cap V(X//R_{V(\sX)}) = \emptyset$ and $V(W//R_{V(\sX)}) \cap V(Y//R_{V(\sX)}) = \emptyset$. Then $\pi^+_W/R_{V(\sX)} (X//R_{V(\sX)}) = \pi^+_W (X)//R_{V(\sX)} = (X \oplus W)//R_{V(\sX)} = X//R_{V(\sX)} \oplus W//R_{V(\sX)} = Y//R_{V(\sX)} \oplus W//R_{V(\sX)} = (Y \oplus W)//R_{V(\sX)} = \pi^+_W (Y)//R_{V(\sX)} = \pi^+_W/R_{V(\sX)} (Y//R_{V(\sX)})$. Therefore $\pi^+_W/R_{V(\sX)}$ is well defined.

To see that $\pi^+_W/R_{V(\sX)} = \pi^+_{W//R_{V(\sX)}}$, let $X//R_{V(\sX)} \in \X/R_{V(\sX)}$. First, if $V(W) \cap V(X) \ne \emptyset$, equivalently $V(W//R_{V(\sX)}) \cap V(X//R_{V(\sX)}) \ne \emptyset$ by Part~\ref{lemma:RequivH1} of Proposition~\ref{lemma:RequivH}, then $\pi^+_W/R_{V(\sX)} (X//R_{V(\sX)}) = \pi^+_W (X)//R_{V(\sX)} = X//R_{V(\sX)} = \pi^+_{W//R_{V(\sX)}} (X//R_{V(\sX)})$. Second, if $V(W) \cap V(X) = \emptyset$, equivalently $V(W//R_{V(\sX)}) \cap V(X//R_{V(\sX)}) = \emptyset$ from Part~\ref{lemma:RequivH1} of Proposition~\ref{lemma:RequivH}, then $\pi^+_W/R_{V(\sX)} (X//R_{V(\sX)}) = \pi^+_W (X)//R_{V(\sX)} = (X \oplus W)//R_{V(\sX)} = X//R_{V(\sX)} \oplus W//R_{V(\sX)} = \pi^+_{W//R_{V(\sX)}} (X//R_{V(\sX)})$. Therefore $\pi^+_W/R_{V(\sX)} = \pi^+_{W//R_{V(\sX)}}$.
\end{proof}

\section{Concluding remarks}
Hypergraphs are of interest within pure mathematics as well as in applications of mathematics, in the latter case because they provide a general framework for modelling higher-order interactions in networks. Hypergraph transformations allow for a formal description of structural modifications of hypergraphs, and in particular can model dynamic properties of networks. Function-based forms of hypergraph transformations are important as they can be incorporated into larger mathematical structures and are readily applicable for modelling the physical world, however no suitable theory of function-based hypergraph transformations exists in the literature.

In this article we present a new general theory for function-based hypergraph transformations which are defined on finite families of finite hypergraphs. Our notion of a hypergraph transformation modifies a hypergraph by replacing certain connected components of the hypergraph, according to the collection of distinguished hypergraphs associated with the transformation. In this way, a given hypergraph transformation replaces the same subset of connected components in any hypergraph in its domain with the same new connected components, thereby ensuring consistency of action.

We establish sufficient conditions for the commutativity of a given set of hypergraph transformations, based on a notion of pairwise disjointness for the transformations. We also demonstrate how a hypergraph transformation can be modified to obtain a new transformation by appropriately modifying the collection of distinguished hypergraphs. Further, since quotient hypergraphs can enable the simplification and comparison of hypergraphs, we consider a notion of a quotient hypergraph transformation.

Finally, to illustrate the general theory we provide specific examples of hypergraph transformations that add or delete a set of hyperedges or a hypergraph, which comprise fundamental transformations of hypergraphs.

\section*{Declaration of competing interest}
The author declares that there is no conflict of interest related to the research presented in this manuscript.

\section*{Data availability}
No data was used for the research described in the manuscript.

\clearpage

\clearpage

\newcommand{\noop}[1]{}


\begin{thebibliography}{10}

\bibitem{Merris1998}
Merris R.
\newblock Laplacian graph eigenvectors.
\newblock \emph{Linear Algebra and its Applications}. 1998;278:221--236.
\newblock doi:10.1016/S0024-3795(97)10080-5.

\bibitem{Pfaltz2018}
Pfaltz JL.
\newblock A graph similarity relation defined by graph transformation.
\newblock In: {2018 International Conference on Applied Mathematics \&
  Computational Science}. Los Alamitos, CA: {IEEE}; 2018. p. 164--170.
\newblock doi:10.1109/ICAMCS.NET46018.2018.00035.

\bibitem{Lozin2011}
Lozin VV.
\newblock Stability preserving transformations of graphs.
\newblock \emph{Annals of Operations Research}. 2011;188:331--341.
\newblock doi:10.1007/s10479-008-0395-1.

\bibitem{Blinov2006}
Blinov ML, Yang J, Faeder JR, Hlavacek WS.
\newblock Graph theory for rule-based modeling of biochemical networks.
\newblock In: {Transactions on Computational Systems Biology {VII}}. Lecture
  Notes in Computer Science. Berlin: Springer; 2006. p. 89--106.
\newblock doi:10.1007/11905455\_5.

\bibitem{Bunimovich2012}
Bunimovich LA, Webb BZ.
\newblock Isospectral graph transformations, spectral equivalence, and global
  stability of dynamical networks.
\newblock \emph{Nonlinearity}. 2012;25:211--254.
\newblock doi:10.1088/0951-7715/25/1/211.

\bibitem{Andersen2013}
Andersen JL, Flamm C, Merkle D, Stadler PF.
\newblock Inferring chemical reaction patterns using rule composition in graph
  grammars.
\newblock \emph{Journal of Systems Chemistry}. 2013;4:4.
\newblock doi:10.1186/1759-2208-4-4.

\bibitem{Voss2023}
Voss C, Petzold F, Rudolph S.
\newblock Graph transformation in engineering design: an overview of the last
  decade.
\newblock \emph{Artificial Intelligence for Engineering Design, Analysis and
  Manufacturing}. 2023;37:e5.
\newblock doi:10.1017/S089006042200018X.

\bibitem{Rozenberg1997}
Rozenberg G.
\newblock {Handbook of Graph Grammars and Computing by Graph Transformation}.
  vol.~1.
\newblock Singapore: World Scientific; 1997.
\newblock doi:10.1142/3303.

\bibitem{Ehrig2006}
Ehrig H, Ehrig K, Prange U, Taentzer G.
\newblock {Fundamentals of Algebraic Graph Transformation}.
\newblock Berlin, Heidelberg: Springer; 2006.
\newblock doi:10.1007/3-540-31188-2.

\bibitem{Heckel2006}
Heckel R.
\newblock Graph transformation in a nutshell.
\newblock \emph{Electronic Notes in Theoretical Computer Science}.
  2006;148:187--198.
\newblock doi:10.1016/j.entcs.2005.12.018.

\bibitem{Klamt2009}
Klamt S, Haus UU, Theis F.
\newblock Hypergraphs and cellular networks.
\newblock \emph{{PLoS} Computational Biology}. 2009;5:e1000385.
\newblock doi:10.1371/journal.pcbi.1000385.

\bibitem{Feng2021}
Feng S, Heath E, Jefferson B, Joslyn C, Kvinge H, Mitchell HD, et~al.
\newblock Hypergraph models of biological networks to identify genes critical
  to pathogenic viral response.
\newblock \emph{{BMC} Bioinformatics}. 2021;22:287.
\newblock doi:10.1186/s12859-021-04197-2.

\bibitem{Rossello2004}
Rossell{\'o} F, Valiente G.
\newblock Analysis of metabolic pathways by graph transformation.
\newblock In: Ehrig H, Engels G, Parisi-Presicce F, Rozenberg G, editors.
  {Graph Transformations}. Berlin: Springer; 2004. p. 70--82.

\bibitem{Yadav2004}
Yadav MK, Kelley BP, Silverman SM.
\newblock The potential of a chemical graph transformation system.
\newblock In: Ehrig H, Engels G, Parisi-Presicce F, Rozenberg G, editors.
  {Graph Transformations}. Berlin: Springer; 2004. p. 83--95.

\bibitem{Rossello2005}
Rossell{\'o} F, Valiente G.
\newblock Graph transformation in molecular biology.
\newblock In: Kreowski HJ, Montanari U, Orejas F, Rozenberg G, Taentzer G,
  editors. {Formal Methods in Software and Systems Modeling}. Berlin: Springer;
  2005. p. 116--133.
\newblock doi:10.1007/978-3-540-31847-7\_7.

\bibitem{Dewar2018}
Dewar M, Pike D, Proos J.
\newblock Connectivity in hypergraphs.
\newblock \emph{Canadian Mathematical Bulletin}. 2018;61:252--271.
\newblock doi:10.4153/CMB-2018-005-9.

\bibitem{Knuth1994}
Knuth DE.
\newblock The sandwich theorem.
\newblock \emph{The Electronic Journal of Combinatorics}. 1994;1:A1.
\newblock doi:10.37236/1193.

\bibitem{Diestel2017}
Diestel R.
\newblock {Graph Theory}.
\newblock 5th ed. Springer Berlin Heidelberg; 2017.
\newblock doi:10.1007/978-3-662-53622-3.

\end{thebibliography}
\end{document}